\documentclass{article}
\usepackage[english]{babel}
\usepackage{graphicx,multicol}
\usepackage{epic,eepic,epsfig}
\usepackage{amssymb}
\usepackage{amsmath,amsthm}
\usepackage{color}
\usepackage{tikz}
\usetikzlibrary{arrows.meta}
\usetikzlibrary {shapes.multipart}
\usepackage{epsfig,url}
\usepackage{multirow}

\usetikzlibrary{calc}

\usepackage{amsfonts}
\usepackage{setspace}
\usepackage{geometry}
\usepackage[T1]{fontenc}
\usepackage{eso-pic} % Magic that gets rid the newline errors in the pics

  \newcommand{\noCol}[1]{}
\newcommand{\fvs}{{\rm fvs}}
\newcommand{\be}{\begin{enumerate}}
\newcommand{\ee}{\end{enumerate}}
\newcommand{\bd}{\begin{description}}
\newcommand{\ed}{\end{description}}

\newcommand{\beq}{\begin{equation}}
\newcommand{\eeq}{\end{equation}}

\newcommand{\2}{\vspace{2mm}}

\renewenvironment{proof}[1][]{\par \noindent {\bf Proof#1}.\ }{\hfill$\Box$
\par \vspace{11pt}}

\newtheorem{theorem}{Theorem}[section]
\newtheorem{lemma}[theorem]{Lemma}

\newtheorem{corollary}[theorem]{Corollary}
\newtheorem{claim}{Claim}

\theoremstyle{definition}

\newtheorem{conjecture}[theorem]{Conjecture}

 \newcommand{\WithoutColor}[1]{}  \newcommand{\redC}[1]{red} \newcommand{\blueC}[1]{blue}  

\begin{document}
\bibliographystyle{plain}
%\pagewiselinenumbering
%\setpagewiselinenumbers
%\modulolinenumbers[1]
%\linenumbers

\title{Feedback vertex sets of digraphs with bounded maximum degree}
\author{Jiangdong Ai\thanks{School of Mathematical Sciences and LPMC, Nankai University. {\tt jd@nankai.edu.cn}.}
\hspace{2mm}
Gregory Gutin\thanks{Department of Computer Science, Royal Holloway University of London, {\tt g.gutin@rhul.ac.uk}, and School of Mathematical Sciences and LPMC, Nankai University.}
\hspace{2mm} 
Xiangzhou Liu\thanks{School of Mathematical Sciences and LPMC, Nankai University. {\tt  
i19991210@163.com}.}
\hspace{2mm} Anders Yeo\thanks{Department of Mathematics and Computer Science, University of Southern Denmark. {\tt yeo@imada.sdu.dk}, and Department of Mathematics, University of Johannesburg.}
\hspace{2mm} Yacong Zhou\thanks{Shenzhen Institutes of Advanced Technology, Chinese Academy of Sciences. {\tt yacong.zhou96@gmail.com}}
}
\date{\today}
\maketitle
\begin{abstract}
A digraph $D$ is an oriented graph if $D$ does not have a pair of opposite arcs. The degree of a vertex $v$ of $D$ is the sum of the in-degree and out-degree of $v.$ Let $\fvs(D)$ be the minimum number of vertices whose deletion from $D$ makes it acyclic. 
Let $D$ be a digraph with $n$ vertices and maximum degree $\Delta$. We prove the following bounds. If $D$ is an oriented graph, then $\fvs(D)\leq \frac{3n}{7}$ when $\Delta\le 4$ and $\fvs(D)\leq \frac{n}{2}$ when $\Delta\le 5$. If $D$ is a connected digraph, $\Delta\le 4$ and $D$ is not obtained from an odd undirected cycle by replacing every edge with the pair of opposite arcs with the same endvertices, then $\fvs(D)\leq \frac{n}{2}$. If $D$ is an arbitrary digraph with $\Delta\le 5$ then $\fvs(D)\leq \frac{2n}{3}.$
Note that all the above bounds are tight.
%We prove that if $\Delta\le 4$ then $\fvs(D)\leq \frac{n}{2}$, and if $D$ is an oriented graph and $\Delta\le 4$ then $\fvs(D)\leq \frac{3n}{7}.$
%If $\Delta\le 5$ then $\fvs(D)\leq \frac{2n}{3}$, and if $D$ is an oriented graph and $\Delta\le 5$, then $\fvs(D)\leq \frac{n}{2}$. 
%In this paper, we prove that for any oriented graph $D$ with maximum degree $4$, $fvs(D)\leq \frac{3n}{7}$; if we allow $2$-cycles, then $fvs(D)\leq \frac{n}{2}$. And, for any oriented graph $D$ with maximum degree $5$, $fvs(D)\leq \frac{n}{2}$; if we allow $2$-cycles, then $fvs(D)\leq \frac{2n}{3}$. We also note that all bounds we mentioned are tight.
\end{abstract}

\section{Introduction}
In this paper, the standard terminology and notation in digraph theory are used, following \cite{BJG}. 
Some important terminology and notation are given at the end of this section.  

A set $F\subseteq V(D)$ of a digraph $D$, is a {\em feedback vertex set} if $D-F$ is acyclic. Similarly, one defines a {\em feedback arc set}.
Deciding whether a graph has a feedback vertex set of a given size is among the 21 original NP-complete problems of Karp~\cite{K72}. Cavallaro and Fluschnik~\cite{Hamiltonian} proved that  the problem remains NP-hard even on subclasses of Hamiltonian graphs. Thus, finding the minimum size of a feedback vertex set of a digraph $D$, denoted by $\fvs(D)$ is a challenging algorithmic problem and was extensively studied in the literature.

In particular, some scholars studied the upper bounds on $\fvs(D)$ for arbitrary digraphs $D$. Akbari, Ghodrati, Jabalameli and Saghafian \cite{AGJS} obtained a lower bound on $|V(D)|-\fvs(D)$ inspired by the well-known Caro-Wei lower bound for the maximum size of an independent set in an undirected graph. This bound was improved in two different ways by Asgarli, Falkenhagen and Hoshi \cite{AFH}.

To obtain tight or close-to-tight lower bounds on $\fvs(D)$, many researchers studied such bounds for special classes of {\em oriented graphs} (or, {\em orgraphs}), i.e., digraphs without opposite arcs. We will first discuss new results obtained in this paper and then we will overview numerous results obtained by other researchers.

Several scholars carried out research on tight upper bounds for the minimum size of a feedback arc set in orgraphs with small maximum degree, see, e.g. \cite{BergerShor1990,BergerShor1997,EadesLinSmyth1993,Hanauer2017Thesis,HanauerBrandenburgAuer2013,gutinFAS}. In particular, the authors of \cite{gutinFAS} obtained tight upper bounds when the maximum degree $\Delta\le 4$ and when $\Delta\le 5.$ The authors of \cite{gutinFAS} proved a conjecture in \cite{Hanauer2017Thesis} and disproved another one in \cite{Hanauer2017Thesis}, both for $\Delta\le 5.$ 
In this paper, similarly, we study tight upper bounds for the minimum size of a feedback vertex set in orgraphs with small bounded maximum degree $\Delta$. Clearly, the case of $\Delta\le 2$ is trivial. In the case of $\Delta\le 3$, the tight bound $\fvs(D) \le \frac{n}{3}$ for oriented graphs was proved in \cite{Hanauer2017Thesis}, and for digraphs the bound $\fvs(D) \le \frac{n}{2}$ follows by induction: delete both vertices of a 2-cycle and use the induction hypothesis, but use only one vertex of the 2-cycle in feedback vertex set (due to possible existence of 2-cycles the bound is tight). 
We establish new tight upper bounds on the minimum feedback vertex set size in directed and oriented graphs with $\Delta\le 4$ and $\Delta\le 5$. Specifically, we prove for $\Delta\le 4$ that any orgraph $D$ on $n$ vertices  satisfies $\fvs(D) \le \frac{3n}{7}$, and if $D$ is a connected digraph not obtained from an odd undirected cycle by replacing every edge with the pair of opposite arcs with the same endvertices, then $\fvs(D) \le \frac{n}{2}$. For $\Delta\le 5$, we show $\fvs(D) \le \frac{n}{2}$ for an orgraph $D$, and $\fvs(D) \le \frac{2n}{3}$ for a digraph $D$. Moreover, we provide explicit constructions that demonstrate that all of these bounds are tight. 

Tournaments form a well-known class of orgraphs. Already Stearns~\cite{S59} and Erd\H{o}s and Moser~\cite{EM64} have shown that for any tournament $T$ on $n$ vertices, we have $\fvs(T)\le n-\lfloor\log_2n\rfloor-1$, while there are tournaments where there is no feedback vertex set with less than $n-2\lfloor\log_2n\rfloor-1$ vertices. More precise bounds for small values of $n$ have been obtained by Sanchez-Flores~\cite{SF94,SF98} and recently more work has been done in that direction by Neiman, Mackey and Heule~\cite{NMH20} and by Lidick\'y and Pfender~\cite{LP2021}. Improving the asymptotic upper and lower bounds remains an open problem.

Another class that has received a lot of attention is planar orgraphs.
Let $D_{g}^{n}$ denote a planar digraph with $n$ vertices and directed girth $g$. Harutyunyan ~\cite{Har11,HM15} conjectured that $\fvs(D_{3}^{n})\le \frac{2n}5$. This conjecture was refuted by Knauer, Valicov and Wenger~\cite{KVW17}, who showed that $\fvs(D_{g}^{n})\ge \frac{n-1}{g-1}$ for all $g\ge 3$. On the other hand,  Golowich and Rolnick~\cite{GR15} established the upper bounds $\fvs(D_{4}^{n})\le \frac{7n}{12}$, $\fvs(D_{5}^{n})\le \frac{8n}{15}$, and $\fvs(D_{g}^{n})\le \frac{3n-6}{g}$
for all $g\ge 6$. Harutyunyan and Mohar~\cite{HM15} proved
that the vertex set of every planar orgraph of girth at least 5 can be partitioned into two acyclic subgraphs. This result was extended to planar orgraphs of girth 4 by Li and Mohar~\cite{LM17}, and therefore $\fvs(D_{4}^{n})\le \frac{n}2$. Further improvements were obtained by Esperet, Lemoine, and Maffray~\cite{ELM17},  who showed $\fvs(D_{4}^{n})\le \tfrac{5n-5}9$, $\fvs(D_{5}^{n})\le \tfrac{2n-5}4$, and $\fvs(D_{g}^{n})\le \tfrac{2n-6}{g}$ if $g\ge 6$. These improved results of Golowich and Rolnick. However, a question of Albertson~\cite{W06,M02} asking whether any planar orgraph has a feedback vertex set on at most half its vertices, remains open. Note that this question is a weakening of the undirected setting as well as of the famous Neumann-Lara conjecture~\cite{NL85}. Further, it is known that if true this bound is best-possible~\cite{KVW17}. Moreover, note that the best known upper bound coincides with the above mentioned $\frac{3}{5}n$ from the undirected setting~\cite{B76}.

Knauer, La and Valicov \cite{KLV} studied the minimum size of a feedback vertex set in oriented and undirected $n$-vertex graphs of given degeneracy or treewidth (for orgraphs, both parameters are undirected, i.e. they are of the underlying graphs of orgraphs). Since in this paper we only study digraphs, we will discuss only the results in \cite{KLV} obtained for orgraphs. For an orgraph $D$ of degeneracy $k\ge 2$, Knauer, La and Valicov proved that $\fvs(D)\le \frac{k-1}{k+1}n$ and that this inequality is strict when $k$ is odd.
To compare our new results with those in \cite{KLV}, note that orgraphs with the maximum degree $\Delta$ are of degeneracy $\Delta$. For $\Delta\le 4,$ while our bound is $\fvs(D) \le \frac{3n}{7}$, the bound of \cite{KLV} is $\fvs(D) \le \frac{3n}{5}.$ For $\Delta\le 5,$ while our bound is $\fvs(D) \le \frac{n}{2}$, the bound of \cite{KLV} is $\fvs(D) \le \frac{2n}{3}.$ Thus, our results for orgraphs with $\Delta\le 4$ and $\Delta\le 5$
do not follow from the relevant results in \cite{KLV}. 
For an orgraph $D$ of treewidth $k\ge 2$, Knauer, La and Valicov proved that $\fvs(D)\le  \frac{k}{k+3}n$. It seems that this bound is unlikely to be tight. 

% For oriented graphs of bounded degeneracy $k$, Knauer, La, and Valicov~\cite{treewidth} prove that $\fvs(D)\leq\frac{k-1}{k+1}n$ and that this inequality is strict when $k$ is odd. For oriented graphs of bounded treewidth $t\geq 2$, Knauer, La, and Valicov~\cite{treewidth} show that $\fvs(D) \leq \frac{t}{t+3}n$.

\paragraph{Some Terminology and Notation}
Let $D$ be a digraph and let $v \in V(D)$. The out-degree
(in-degree, respectively) is denoted $d^+_D(v)$ ($d^-_D(v)$, respectively). Recall that the {\em degree} of $v$ is $d_D(v) = d^+_D(v) + d^-_D(v)$. The maximum degree $\Delta(D)$ of $D$ is defined as $\Delta(D) = \max_{v \in V(D)} d_D(v)$. A digraph $D$ is {\em $k$-regular} if $d^+_D(v) =d^-_D(v)=k$ for every vertex $v\in V(D)$. 

The {\em order} of a directed or undirected graph $H$ is the number of vertices in $H$.	In a digraph, a {\em cycle} ({\em path}, respectively)
is a directed cycle (directed path, respectively). 	A {\em $k$-cycle} is a
cycle with $k$ vertices. The {\em directed girth} of a digraph $D$ is the minimum $k$ for which $D$ has a $k$-cycle. In particular, an orgraph $D$ is of directed girth at least 3. 

The {\em underlying graph} of a digraph $D$, is
the undirected graph $UG(D)$  with the same vertex set as $D$ and such that
a pair $u,v$ of distinct vertices are adjacent in $UG(D)$ if there is an arc
between $u$ and $v$ in $D.$ A component of $UG(D)$ is a {\em component} of
$D$ and $D$ is {\em connected} if $UG(D)$ is connected.

\section{Oriented and Directed Graphs with $\Delta\leq 4$}

In Subsections \ref{4orient} and \ref{4direct}, we obtain sharp upper bounds for oriented and directed graphs, respectively, with maximum degree at most 4.

\subsection{Orgraphs}\label{4orient}

We will define a class ${\cal H}$ of digraphs  as follows.  Let $H_1$ be the $3$-cycle (i.e. $V(H_1)=\{x,y,z\}$ and $A(H_1)=\{xy,yz,zx\}$), let $H_2$ be defined by 
$V(H_2)=\{x,y,z,a,b\}$ and $A(H_2)=\{xy,yz,zx,ab,bx,xa,bz,za\}$, and let $H_3$ be defined by
$V(H_3)=\{x,y,z,a,b\}$ and $A(H_3)=\{xy,yz,zx,ab,bx,ya,bz,za\}$ (see Figure~\ref{H-digraphs}).  Let ${\cal H}$ contain the digraphs $H_1$, $H_2$, $H_3$ and all digraphs that can be obtained by taking two 
digraphs in ${\cal H}$ (they may or may not be the same digraph) and adding an arc between them. The class ${\cal H}$ contains exactly the digraphs that can be obtained
by this process (see Figure~\ref{H-example}).

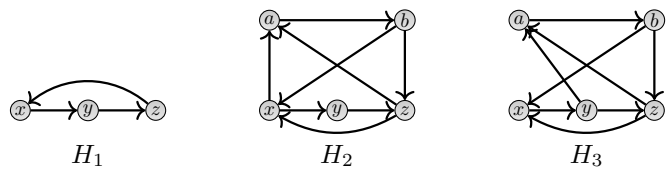
\begin{figure}[htb]
\begin{center}
\tikzstyle{vertexX}=[circle, draw, top color=black!20, bottom color=black!10, minimum size=10pt, scale=0.75, inner sep=0.8pt]
\begin{tikzpicture}[scale=0.3]
  \node at (4,-1) {$H_1$};
  \node (x) at (1,1) [vertexX]{$x$}; 
  \node (y) at (4,1) [vertexX]{$y$};    
  \node (z) at (7,1) [vertexX]{$z$};    
  \draw[->, line width=0.03cm] (x) -- (y);
  \draw[->, line width=0.03cm] (y) -- (z);
  \draw[->, line width=0.03cm] (z) to [out=140, in=40] (x);
\end{tikzpicture} \hspace{1cm}
\begin{tikzpicture}[scale=0.3]
  \node at (4,-1) {$H_2$};
  \node (x) at (1,1) [vertexX]{$x$};
  \node (y) at (4,1) [vertexX]{$y$};
  \node (z) at (7,1) [vertexX]{$z$};

  \node (a) at (1,5) [vertexX]{$a$};
  \node (b) at (7,5) [vertexX]{$b$};
  \draw[->, line width=0.03cm] (x) -- (y);
  \draw[->, line width=0.03cm] (y) -- (z);
  \draw[->, line width=0.03cm] (z) to [out=210, in=330] (x);
  \draw[->, line width=0.03cm] (a) -- (b);
  \draw[->, line width=0.03cm] (b) -- (x);
  \draw[->, line width=0.03cm] (b) -- (z);
  \draw[->, line width=0.03cm] (x) -- (a);
  \draw[->, line width=0.03cm] (z) -- (a);
\end{tikzpicture} \hspace{1cm}
\begin{tikzpicture}[scale=0.3]
  \node at (4,-1) {$H_3$};
  \node (x) at (1,1) [vertexX]{$x$};
  \node (y) at (4,1) [vertexX]{$y$};
  \node (z) at (7,1) [vertexX]{$z$};

  \node (a) at (1,5) [vertexX]{$a$};
  \node (b) at (7,5) [vertexX]{$b$};
  \draw[->, line width=0.03cm] (x) -- (y);
  \draw[->, line width=0.03cm] (y) -- (z);
  \draw[->, line width=0.03cm] (z) to [out=210, in=330] (x);
  \draw[->, line width=0.03cm] (a) -- (b);
  \draw[->, line width=0.03cm] (b) -- (x);
  \draw[->, line width=0.03cm] (b) -- (z);
  \draw[->, line width=0.03cm] (y) -- (a);
  \draw[->, line width=0.03cm] (z) -- (a);
\end{tikzpicture} 
\caption{The digraphs $H_1$, $H_2$ and $H_3$.}
\label{H-digraphs}
\end{center}
\end{figure}

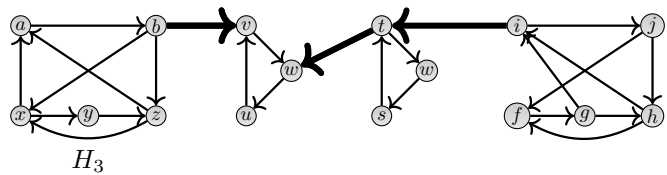
\begin{figure}[htb]
\begin{center}
\tikzstyle{vertexX}=[circle, draw, top color=black!20, bottom color=black!10, minimum size=10pt, scale=0.75, inner sep=0.8pt]
\begin{tikzpicture}[scale=0.3]
%  \node at (10,-1) {\mbox{digraph in ${\cal H}$}};
  \node (x) at (1,1) [vertexX]{$x$};
  \node (y) at (4,1) [vertexX]{$y$};
  \node (z) at (7,1) [vertexX]{$z$};

  \node (a) at (1,5) [vertexX]{$a$};
  \node (b) at (7,5) [vertexX]{$b$};
  \draw[->, line width=0.03cm] (x) -- (y);
  \draw[->, line width=0.03cm] (y) -- (z);
  \draw[->, line width=0.03cm] (z) to [out=210, in=330] (x);
  \draw[->, line width=0.03cm] (a) -- (b);
  \draw[->, line width=0.03cm] (b) -- (x);
  \draw[->, line width=0.03cm] (b) -- (z);
  \draw[->, line width=0.03cm] (x) -- (a);
  \draw[->, line width=0.03cm] (z) -- (a);

  \node (x1) at (11,1) [vertexX]{$u$};
  \node (y1) at (11,5) [vertexX]{$v$};
  \node (z1) at (13,3) [vertexX]{$w$};
  \draw[->, line width=0.03cm] (x1) -- (y1);
  \draw[->, line width=0.03cm] (y1) -- (z1);
  \draw[->, line width=0.03cm] (z1) -- (x1);

  \node (x2) at (17,1) [vertexX]{$s$};
  \node (y2) at (17,5) [vertexX]{$t$};
  \node (z2) at (19,3) [vertexX]{$w$};
  \draw[->, line width=0.03cm] (x2) -- (y2);
  \draw[->, line width=0.03cm] (y2) -- (z2);
  \draw[->, line width=0.03cm] (z2) -- (x2);

  \draw[->, line width=0.08cm] (b) -- (y1);
  \draw[->, line width=0.08cm] (y2) -- (z1);

  \node at (4,-1) {$H_3$};
  \node (x3) at (23,1) [vertexX]{$f$};
  \node (y3) at (26,1) [vertexX]{$g$};
  \node (z3) at (29,1) [vertexX]{$h$};

  \node (a3) at (23,5) [vertexX]{$i$};
  \node (b3) at (29,5) [vertexX]{$j$};
  \draw[->, line width=0.03cm] (x3) -- (y3);
  \draw[->, line width=0.03cm] (y3) -- (z3);
  \draw[->, line width=0.03cm] (z3) to [out=210, in=330] (x3);
  \draw[->, line width=0.03cm] (a3) -- (b3);
  \draw[->, line width=0.03cm] (b3) -- (x3);
  \draw[->, line width=0.03cm] (b3) -- (z3); 
  \draw[->, line width=0.03cm] (y3) -- (a3);
  \draw[->, line width=0.03cm] (z3) -- (a3);

  \draw[->, line width=0.08cm] (b) -- (y1);
  \draw[->, line width=0.08cm] (y2) -- (z1);
  \draw[->, line width=0.08cm] (a3) -- (y2);

\end{tikzpicture} 
\caption{A digraph in ${\cal H}$ (where the thick arcs are bridges).}
\label{H-example}
\end{center}
\end{figure}

For a digraph $D$, let $h(D)$ denote the number of connected components in $D$ that belong to ${\cal H}$. We will show the following theorem.
\begin{theorem} \label{main4}
  Let $D$ be an oriented graph with maximum degree at most $4$. Then, we have
\[
\fvs(D) \leq \frac{|V(D)|+|A(D)| + h(D)}{7}
\]
\end{theorem}

Theorem~\ref{main4} implies the following: 

\begin{corollary} \label{mainCor}
  If $D$ is an oriented graph of order $n$ with the maximum degree at most $4$, then $\fvs(D) \leq \frac{3n}{7}$.
\end{corollary}
\begin{proof}
Note that $|A(D)| \leq 2|V(D)|$ for all digraphs $D$ of maximum degree at most $4$, $|A(D)| < 2|V(D)|$ 
for all digraphs of maximum degree at most $4$ in which at least one vertex is of degree at most 3, and every digraph in $\cal H$ has a vertex of degree at most 3. 
Thus, Theorem~\ref{main4} implies that for every connected oriented graph of order $n$, we have $\fvs(D) \leq \frac{3n}{7}$. Summing such bounds over all connected components of an oriented graph $D$, we obtain the bound of the corollary.
%$\fvs(D) \leq \frac{3n}{7}$ 
\end{proof}
%\subsection{Proof of Theorem~\ref{main4}}

This corollary is tight due to a $2$-regular digraph of order $7$ with minimum feedback vertex set equal to 3, see Figure~\ref{orgraph4}.
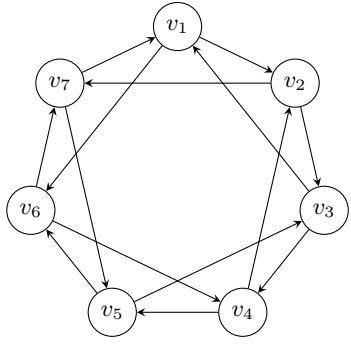
\begin{figure}[htb]
\begin{center}
\tikzstyle{vertexX}=[circle, draw, top color=black!20, bottom color=black!10, minimum size=10pt, scale=0.75, inner sep=0.8pt]
\begin{tikzpicture}[>=stealth,
                    node style/.style={circle, draw, fill=white, inner sep=3pt, font=\small}]
\node[node style] (v1) at (90:2) {$v_1$};
\node[node style] (v2) at (90-51.43:2) {$v_2$};
\node[node style] (v3) at (90-102.86:2) {$v_3$};
\node[node style] (v4) at (90-154.29:2) {$v_4$};
\node[node style] (v5) at (90-205.71:2) {$v_5$};
\node[node style] (v6) at (90-257.14:2) {$v_6$};
\node[node style] (v7) at (90-308.57:2) {$v_7$};

\draw[->] (v1) -- (v2);
\draw[->] (v2) -- (v3);
\draw[->] (v3) -- (v4);
\draw[->] (v4) -- (v5);
\draw[->] (v5) -- (v6);
\draw[->] (v6) -- (v7);
\draw[->] (v7) -- (v1);
\draw[->] (v4) -- (v2); % 4-2
\draw[->] (v2) -- (v7); % 2-7
\draw[->] (v7) -- (v5); % 7-5
\draw[->] (v3) -- (v1); % 3-1
\draw[->] (v5) -- (v3); % 5-3
\draw[->] (v6) -- (v4); % 4-6
\draw[->] (v1) -- (v6); % 1-6

\end{tikzpicture}
\caption{a $2$-regular digraph of order $7$ with feedback vertex set equal to 3}
\label{orgraph4}
\end{center}
\end{figure}

Before proving Theorem~\ref{main4} we need the following lemma.

\begin{lemma} \label{Hlemma}
  Let $H \in {\cal H}$ be arbitrary. Then the following statements hold:
\begin{description}
\item[(a):] $\fvs(H)= \frac{n+m+1}{7}$.
\item[(b):] Let $C_1,C_2,\ldots,C_l$ be the strong components of $H$ and let $X \subseteq V(H)$ be any set such that $|V(C_i) \cap X| \leq 1$ for all $i=1,2,\ldots,l$.
Then $C_i$ is isomorphic to either $H_1$, $H_2$ or $H_3$ for all $i=1,2,\ldots,l$ and there exists a minimum feedback vertex set in $H$ containing all vertices in $X$.
\item[(c):] If $H \not\in \{H_2,H_3\}$ then $\sum_{x \in V(H)} (4-d(x)) \geq 6$.
\item[(d):] If $H \in \{H_2,H_3\}$ then $\sum_{x \in V(H)} (4-d(x)) =4$.
\end{description}
\end{lemma}
\begin{proof}
We prove (a) by induction on $|V(H)|$. Observe first that $\fvs(H_1)=1=\frac{|V(H_1)|+|A(H_1)|+1}{7}$. As $xyzx$ forms a directed 3-cycle in $H_2$ and $H_3$, every feedback vertex set must contain at least one vertex of $\{x,y,z\}$. Hence
	\[
	\fvs(H_i)=2=\frac{|V(H_i)|+|A(H_i)|+1}{7} \qquad \text{for } i\in\{2,3\},
	\]
	since $H_2$ and $H_3$ still contain a directed cycle after the removal of any single vertex of $\{x,y,z\}$, whereas the induced subgraph $H_i-\{x,z\}$ is acyclic for $i\in\{2,3\}$. This settles the base cases of (a). Let $H\in\mathcal{H}$ and assume that $H$ is obtained by adding an arc between $H'$ and $H''$, where $H',H''\in\mathcal{H}$. By the induction hypothesis, we have that
	$
	\fvs(H')=\frac{|V(H')|+|A(H')|+1}{7}$ and $
	\fvs(H'')=\frac{|V(H'')|+|A(H'')|+1}{7}.
	$
	Since the union of minimum feedback vertex sets of $H'$ and $H''$ is clearly a minimum feedback vertex set of $H$ and $|A(H)|=|A(H')|+|A(H'')|+1$, we obtain
	\[
	\fvs(H)
	=\frac{|V(H')|+|A(H')|+|V(H'')|+|A(H'')|+2}{7}
	=\frac{|V(H)|+|A(H)|+1}{7},
	\]
	which proves~(a).
	
By the definition of $\mathcal{H}$, every strong component of $H$ is isomorphic to $H_1$, $H_2$, or $H_3$. Since $\fvs(H_1)=1$ and $\fvs(H_2)=\fvs(H_3)=2$, to prove (b) it suffices to show that every vertex of $H_2$ or $H_3$ is contained in a feedback vertex set of size~2.

Observe that
\[
\{x,b\}, \qquad
\{z,a\}, \qquad
\{y,a\}
\]
are minimum feedback vertex sets of $H_2$ and $H_3$. Hence (b) follows, as each vertex of $C_i$ lies in at least one of these sets.

	(d) is straightforward to verify. It remains to prove (c). Note that
	\[
	\sum_{x\in V(H_1)} (4-d_{H_1}(x)) = 6.
	\]
	Hence we may assume that $H$ has at least two strong components. Let $C_1$ and $C_\ell$ be the two end strong components, i.e., those adjacent to exactly one other strong component. Using (d), the identity
	$
	\sum_{x\in V(H_1)} (4-d_{H_1}(x)) = 6,
	$
	and the fact that there is exactly one arc between $C_1$ (or $C_\ell$) and another component, we obtain
	\[
	\sum_{x\in V(C_1)\cup V(C_l)} (4-d_D(x))=	\sum_{x\in V(C_1)} (4-d_{C_1}(x))+\sum_{x\in V(C_l)} (4-d_{C_l}(x))-2  \ge 4 + 4 - 2 = 6.
	\]
	This establishes (c) and completes the proof of the lemma.
\end{proof}

Let us recall Theorem~\ref{main4}.

\2

\noindent
{\bf Theorem~\ref{main4}.} {\em Let $D$ be an oriented graph with the maximum degree at most $4$. Then, we have %, the following holds:
\[
\fvs(D) \leq \frac{|V(D)|+|A(D)| + h(D)}{7}
\]
}
\begin{proof}
Assume that the theorem is false and let $D$ be a counter example to the theorem of smallest possible order. We will prove a number of claims regarding $D$, ending in a contradiction, which will complete 
the proof. Let $D$ have order $n$ and size $m$.

\2

{\bf Claim A} {\em $D$ is connected and $h(D)=0$.  In fact, $UG(D)$ is even $2$-edge-connected.}

\2

{\bf Proof of Claim~A.} For the sake of contradiction, assume $D$ is not connected. By the minimality of $n$ we note that the theorem holds for all components of $D$ and therefore also for $D$, a contradiction.
So $D$ is connected. For the sake of contradiction, assume that $h(D) > 0$. As $D$ is connected, this implies that $h(D)=1$ and $D \in {\cal H}$. 
By Lemma~\ref{Hlemma}(a) we note that $\fvs(D)= \frac{n+m+1}{7}$ and therefore $D$ is not a counter example to the theorem, a contradiction. 

Now we prove the next part of the claim.
For the sake of contradiction, assume that $D$ contains an arc $a \in A(D)$ such that $D-a$ is not connected. 
Analogously to the above, the theorem holds for all components of $D-a$ and therefore also for $D$, as $a$ does not belong to any cycle of $D$, a contradiction. So Claim~A holds.

\2

{\bf Claim B} {\em $n \geq 4$.}

\2

{\bf Proof of Claim~B.} For the sake of contradiction, assume that $n \leq 3$. 
As $D \not\in {\cal H}$ by Claim~A, we note that $D$ is not a $3$-cycle. Therefore, $\fvs(D)=0$ and  $D$ is not a counter example to the theorem, a contradiction. So Claim~B holds.

\2

{\bf Claim C} {\em If $H$ is an induced subgraph of $D$ and $H \in {\cal H}$, then there are at least 3 arcs between $V(H)$ and $V(D) \setminus V(H)$.}

\2

{\bf Proof of Claim~C.}  Let $H$ be an induced subgraph of $D$ and, for the sake of contradiction, assume that $H \in {\cal H}$, but
there are at most 2  arcs between $V(H)$ and $V(D) \setminus V(H)$.  By Claim~A we note that there must be exactly 2  arcs between $V(H)$ and $V(D) \setminus V(H)$ and 
let $a$ be such an arc. Let $x$ be the endpoint of $a$ that belongs to $V(H)$. 
By Lemma~\ref{Hlemma}(b), there exists a minimum feedback vertex set $F_H$ in $H$ that contains $x$.

Let $D'=D - V(H)$. We must have $h(D') \leq 1$ by Claim~A. 
As $n'=n-|V(H)|$ and $m'=m-|A(H)|-2$ and $h(D') \leq 1$ 
the following holds by induction, thereby proving Claim~C.

\[
\begin{array}{rcl} \vspace{0.1cm}
\fvs(D) & \leq  & \fvs(D') + |F_H| \\ \vspace{0.1cm}
       & \leq & \frac{n'+m'+h(D')}{7} + \frac{|V(H)|+|A(H)|+1}{7} \\ \vspace{0.1cm}
       & \leq  & \frac{(n-|V(H)|)+(m-|A(H)|-2)+1}{7} + \frac{|V(H)|+|A(H)|+1}{7}  \\
       & \leq & \frac{n+m}{7} \\
\end{array}
\]

\2

{\bf Claim D} {\em $d^+(x) \geq 1$ and $d^-(x) \geq 1$ for all $x \in V(D)$.}

\2

{\bf Proof of Claim~D.} By Claim~A and Claim~B we note that $D$ contains no isolated vertices.
For the sake of contradiction assume that $x \in V(D)$ has in-degree or out-degree zero. 
Let $D'=D-x$ and let $n'$ be the order of $D'$ and let $m'$ be the size of $D'$.
 If $F'$ is a minimum feedback vertex set in $D'$ then $F'$ is also a 
feedback vertex set in $D$ (as $x$ doesn't lie on any cycle in $D$), which implies that $\fvs(D) \leq \fvs(D')$.

As $n'=n-1$ and $m'=m-d(x)$ and $h(D') \leq 1$ (by Claim~C, as $h(D)=0$ by Claim~A) 
the following holds by induction, thereby proving Claim~D.

\[
\begin{array}{rcl} \vspace{0.1cm}
\fvs(D) & \leq  & \fvs(D') \\ \vspace{0.1cm}
       & \leq & \frac{n'+m'+h(D')}{7} \\ \vspace{0.1cm}
       & \leq  & \frac{(n-1)+(m-d(x))+1}{7} \\
       &  < & \frac{n+m}{7} \\
\end{array}
\]

\2

{\bf Claim E} {\em $\delta(D) \geq 3$.}

\2
  
{\bf Proof of Claim~E.} For the sake of contradiction, assume that $\delta(D) \leq 2$. Let $x \in V(D)$ be chosen such that $d(x)$ is minimum. By Claim~D we note that 
 $d^+(x) = d^-(x) = 1$.
Let $N^-(x)=\{y\}$ and let $N^+(x)=\{z\}$ and note that $yxz$ is a path in $D$.  We now prove the following subclaims.

\2

{\bf Subclaim E.1. $zy \in A(D)$:} 

\2

{\bf Proof of Subclaim~E.1.} For the sake of contradiction assume that $zy \not\in A(D)$.
Let $D'$ be obtained from $D-x$ by adding the arc $yz$, if $yz$ was not already an arc in $D$.  Let $n'$ be the order of $D'$ and let $m'$ be the size of $D'$.
 If $F'$ is a minimum feedback vertex set in $D'$ then $F'$ is also a
feedback vertex set in $D$, since any cycle in $D-F'$ must use $x$ and therefore substituting the path $yxz$ by the arc $yz$ on this cycle gives us a cycle in $D'-F'$, a contradiction.
Therefore, $\fvs(D) \leq \fvs(D')$. As $D$ is connected we note that $D'$ is connected.
This implies that $h(D') \leq 1$.
The following now holds, by induction, which completes the proof of Subclaim~E.1.

\[
\fvs(D) \leq \fvs(D') \leq \frac{n'+m'+h(D')}{7} \leq \frac{(n-1)+(m-1)+1}{7} < \frac{n+m}{7}
\]

\2

{\bf Subclaim E.2. $d^+(y) \not= 2$ or $d^-(y) \not= 2$ or $d^+(z) \not= 2$ or $d^-(z) \not= 2$:}

\2

{\bf Proof of Subclaim~E.2.} For the sake of contradiction assume that $d^+(y) = 2$ and $d^-(y) = 2$ and $d^+(z) = 2$ and $d^-(z) = 2$. 
By Subclaim E.1. $zy \in A(D)$ (see Figure~\ref{E2yxz}).
Let $D_{yx}=D -\{y,x\}$ and let $D_{xz} = D - \{x,z\}$. For the sake of contradiction assume that $h(D_{yx})=0$. We now
obtain the following contradiction by induction.

\begin{figure}[htb]
\begin{center}
\tikzstyle{vertexX}=[circle, draw, top color=black!20, bottom color=black!10, minimum size=13pt, scale=0.75, inner sep=0.8pt]
\tikzstyle{vertexXp}=[circle, draw, top color=black!20, bottom color=black!10, minimum size=13pt, scale=0.65, inner sep=0.8pt]
\tikzstyle{vertexXpp}=[circle, draw, top color=black!20, bottom color=black!10, minimum size=13pt, scale=0.56, inner sep=0.8pt]
\begin{tikzpicture}[scale=0.3]
%  \node at (4,-1) {$H_2$};
  \node (y) at (1,1) [vertexX]{$y$};
  \node (x) at (5,1) [vertexX]{$x$};
  \node (z) at (9,1) [vertexX]{$z$};

  \node (zp) at (12.5,2) [vertexX]{$z'$};
  \node (zpp) at (12.5,0) [vertexX]{$z''$};

  \draw[->, line width=0.03cm] (y) -- (x);
  \draw[->, line width=0.03cm] (x) -- (z);
  \draw[->, line width=0.03cm] (z) to [out=220, in=320] (y);

  \draw[->, line width=0.03cm, dotted] (y) -- (-1,0.5);
  \draw[->, line width=0.03cm, dotted] (-1,1.5) -- (y);

  \draw[->, line width=0.03cm] (zpp) -- (z);
  \draw[->, line width=0.03cm] (z) -- (zp);

%  \draw[->, line width=0.03cm, dotted] (z) -- (11,0.5);
%  \draw[->, line width=0.03cm, dotted] (11,1.5) -- (z);

\end{tikzpicture} 
\caption{A subgraph of $D$ in Case E.2.}
\label{E2yxz}
\end{center}
\end{figure}
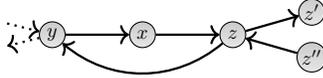

%rrrrrrrrr

\[
\fvs(D)  \leq  \fvs(D_{yx}) + |\{y\}| 
        \leq  \frac{(n-2)+(m-5)+h(D_{yx})}{7} + 1 
        \leq  \frac{n+m}{7} 
\]

Therefore, $h(D_{yx})>0$. By Claim~C we note that $h(D_{yx})=1$, so let $H_{yx}$ be the component in $D_{yx}$ that belongs to ${\cal H}$.
By Claim~A and Claim~C we note that $H_{yx}=D_{yx}$. So $D_{yx} \in {\cal H}$.
Analogously, we can show that $D_{xz} \in {\cal H}$.

% Let $C_{yx}$ be the strong component in $D_{yx}$ that contains the vertex $z$. Note that in $D_{yx}$ the degree of $z$ is exactly two,
% which implies that $z$ is part of a $3$-cycle, say $z z' z'' z$, in $C_{yx}$ (as this holds for all $H_1$, $H_2$ and $H_3$).

% We first consider the case when $C_{yx}=H_1$, then it is not difficult to see that $z' z'' y z'$ is a strong component in $D_{xz}$.
% And now $\{y,x,z,z',z''\}$ induces a copy of $H_2$ in $D$. And this implies that $h(D) >0$, contradicting Claim~A.
% As $H_3$ has no degree two vertices, we note that $C_{yx} \not= H_3$.
% So we must have $C_{yx}=H_2$. The digraph $C_{yx}-z$ is depicted in Figure~\ref{E2}.
Let $C_{yx}$ be the strong component in $D_{yx}$ that contains the vertex $z$. Note that in $D_{yx}$ the degree of $z$ is exactly two, which implies that $C_{yx} \not= H_3$ (as $H_3$ has no degree two vertices). Thus, $z$ is part of a $3$-cycle, say $z z' z'' z$, in $C_{yx}$ (as this holds for all $H_1$ and $H_2$).
We first consider the case when $C_{yx}=H_1$, then it is not difficult to see that $z' z'' y z'$ is a strong component in $D_{xz}$. And now $\{y,x,z,z',z''\}$ induces a copy of $H_2$ in $D$. And this implies that $h(D) >0$, contradicting Claim~A. So we must have $C_{yx}=H_2$. The digraph $C_{yx}-z$ is depicted in Figure~\ref{E2}.

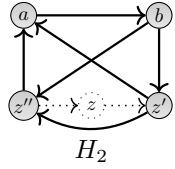
\begin{figure}[htb]
\begin{center}
\tikzstyle{vertexX}=[circle, draw, top color=black!20, bottom color=black!10, minimum size=13pt, scale=0.75, inner sep=0.8pt]
\tikzstyle{vertexZ}=[circle, dotted, draw, minimum size=13pt, scale=0.75, inner sep=0.8pt]

\begin{tikzpicture}[scale=0.3]
  \node at (4,-1) {$H_2$};
  \node (x) at (1,1) [vertexX]{$z''$};
  \node (y) at (4,1) [vertexZ]{$z$};
  \node (z) at (7,1) [vertexX]{$z'$};

  \node (a) at (1,5) [vertexX]{$a$};
  \node (b) at (7,5) [vertexX]{$b$};
  \draw[->, line width=0.02cm, dotted] (x) -- (y);
  \draw[->, line width=0.02cm, dotted] (y) -- (z);
  \draw[->, line width=0.03cm] (z) to [out=210, in=330] (x);
  \draw[->, line width=0.03cm] (a) -- (b);
  \draw[->, line width=0.03cm] (b) -- (x);
  \draw[->, line width=0.03cm] (b) -- (z);
  \draw[->, line width=0.03cm] (x) -- (a);
  \draw[->, line width=0.03cm] (z) -- (a);
\end{tikzpicture} \hspace{1cm}
\caption{The digraph $C_{yx}-z$ in Case E.2.}
\label{E2}
\end{center}
\end{figure}

Some strong component, $C^*$, of $D_{xz}$ must contain $z'$, $z''$ and $y$, as $C_{yx}-z$ is a strong component in $D-\{yxz\}$ that does not belong to ${\cal H}$.
Considering $C_{yx}-z$ (depicted in Figure~\ref{E2}) we note that $y$ must be adjacent to $a$ and $b$ (as if $y$ is adjacent to either $z'$ or $z''$ then the degrees would become 
larger than four). But now $C^*$ is not isomorphic to $H_1$, $H_2$ or $H_3$, a contradiction.
This completes the proof of Subclaim~E.2.

\2

We now return to the proof of Claim~E. 
By Subclaim~E.1, we know that $zy \in A(D)$.  
By Subclaim~E.2, we may assume that $d^+(y) \not= 2$ or $d^-(y) \not= 2$ (the case when $d^+(z) \not= 2$ or $d^-(z) \not= 2$ can be handled analogously).
Let $D' = D - \{y,x,z\}$ and let $n'$ be the order of $D'$ and let $m'$ be the size of $D'$.
If $F'$ is a minimum feedback vertex set in $D'$ then
$F' \cup \{z\}$ is a feedback vertex set in $D$. Therefore, $\fvs(D) \leq \fvs(D')+1$. 

Let $s = (d(y)-2)+(d(z)-2)$ and note that $n'=n-3$ and $m'=m-3-s$.
By Claim~C we note that $h(D') \leq 1$ (as $s\leq 4$) and $h(D')=0$ when $s \leq 2$.
We now obtain the following contradiction, by induction.

\[
\fvs(D)  \leq  \fvs(D')+1  
        \leq  \frac{n'+m'+h(D')}{7} + 1  
        \leq  \frac{(n-3)+(m-3-s)+h(D') }{7} + 1  
\]

So, if $s \leq 2$ then $\fvs(D) \leq \frac{n+m}{7}$ (as $s \geq 1$ and $h(D')=0$) and if  $s > 2$ then again $\fvs(D) \leq \frac{n+m}{7}$ (as $h(D') \leq 1$).
So in all cases we have $\fvs(D) \leq \frac{n+m}{7}$, a contradiction.

\2

{\bf Claim F} {\em $D$ is $2$-regular.}

\2
 
{\bf Proof of Claim~F.} For the sake of contradiction assume that $D$ is not $2$-regular, which implies that some vertex, $x$, has $d^+(x)<2$ or $d^-(x)<2$. 
Without loss of generality assume that $d^-(x)<2$. By Claim~D we note that $d^-(x)=1$. By Claim~E we note that $d^+(x) \geq 2$. 
Let $y$ be the unique in-neighbour of $x$ in $D$ and note that $d(y) \geq 3$ by Claim~E. 

Let $D' = D - \{x,y\}$ and let $n'$ be the order of $D'$ and let $m'$ be the size of $D'$.
First consider the case when $h(D')=0$. Note that $\fvs(D) \leq \fvs(D')+1$ as adding $y$  to any feedback vertex set in $D'$ gives us a feedback vertex set in $D$.
Since $n'=n-2$ and $m' \leq m-5$ we now obtain the following contradiction, by induction.

\[
\begin{array}{rcl} \vspace{0.1cm}
\fvs(D) & \leq & \fvs(D')+1  \\ \vspace{0.1cm}
       & \leq & \frac{n'+m'+h(D')}{7} + 1 \\ \vspace{0.1cm}
       & \leq & \frac{(n-2)+(m-5) }{7} + 1 \\ 
       &  =   & \frac{n+m}{7} \\
\end{array}
\]

So we may assume that $h(D')>0$. 
If $d(y)+d(x) \leq 7$, then $h(D') = 1$, by Claim~C and if $d(y)+d(x) =6$, then $D' \in {\cal H}$, by Claim~A and Claim~C.
 We now consider the following cases.

\begin{description}
 \item[Case F.1. $d(y)=4$ and $d(x)=4$:] By Claim~C note that $h(D') \leq 2$. If $F'$ is a minimum feedback vertex set in $D'$ then
$F' \cup \{y\}$ is a feedback vertex set in $D$. Therefore, $\fvs(D) \leq \fvs(D')+1$. We now obtain the following contradiction, by induction.

\[
\begin{array}{rcl} \vspace{0.1cm}
\fvs(D) & \leq & \fvs(D')+1  \\ \vspace{0.1cm}
       & \leq & \frac{n'+m'+h(D')}{7} + 1  \\ \vspace{0.1cm}
       & \leq & \frac{(n-2)+(m-7) +2}{7} + 1  \\ 
       & \leq & \frac{n+m}{7} \\
\end{array}
\]

 \item[Case F.2. $d(y)+d(x)=7$:] Recall that $h(D') = 1$ (by Claim~C). We now, analogously to Case F.1., obtain the following contradiction, by induction.

\[
\begin{array}{rcl} \vspace{0.1cm}
\fvs(D) & \leq & \fvs(D')+1  \\ \vspace{0.1cm}
       & \leq & \frac{n'+m'+h(D')}{7} + 1  \\ \vspace{0.1cm}
       & \leq & \frac{(n-2)+(m-6) +1}{7} + 1  \\ 
       & \leq & \frac{n+m}{7} \\
\end{array}
\]

\begin{figure}[htb]
\begin{center}
\tikzstyle{vertexX}=[circle, draw, top color=black!20, bottom color=black!10, minimum size=13pt, scale=0.75, inner sep=0.8pt]
\tikzstyle{vertexXp}=[circle, draw, top color=black!20, bottom color=black!10, minimum size=13pt, scale=0.65, inner sep=0.8pt]
\tikzstyle{vertexXpp}=[circle, draw, top color=black!20, bottom color=black!10, minimum size=13pt, scale=0.56, inner sep=0.8pt]
\begin{tikzpicture}[scale=0.3]
  \node at (7,-1) {(a)};
  \node (y) at (5,1) [vertexX]{$y$};
  \node (x) at (9,1) [vertexX]{$x$};

  \node (y1) at (1,2) [vertexX]{};
  \node (y2) at (1,0) [vertexX]{};

  \node (x1) at (13,2) [vertexX]{};
  \node (x2) at (13,0) [vertexX]{};

%  \draw[->, line width=0.03cm, dotted] (x) -- (1,2);
  \draw[->, line width=0.03cm] (y1) -- (y);
  \draw[->, line width=0.03cm] (y2) -- (y);

  \draw[->, line width=0.03cm] (y) -- (x);
  \draw[->, line width=0.03cm] (x) -- (x1);
  \draw[->, line width=0.03cm] (x) -- (x2);
%  \draw[->, line width=0.03cm, dotted] (x) -- (13,2);

%  \draw[->, line width=0.03cm] (z) to [out=220, in=320] (y);
\end{tikzpicture}  \hspace{1cm}
\begin{tikzpicture}[scale=0.3]
  \node at (7,-1) {(b)};
  \node (y) at (5,1) [vertexX]{$y$};
  \node (x) at (9,1) [vertexX]{$x$}; 

  \node (y2) at (1,1) [vertexX]{}; 
  \node (z) at (7,5) [vertexX]{$z$};
  \node (x2) at (13,1) [vertexX]{};

%  \draw[->, line width=0.03cm, dotted] (x) -- (1,2);
  \draw[->, line width=0.03cm] (z) -- (y);
  \draw[->, line width=0.03cm] (y2) -- (y);

  \draw[->, line width=0.03cm] (y) -- (x);
  \draw[->, line width=0.03cm] (x) -- (z);
  \draw[->, line width=0.03cm] (x) -- (x2);
%  \draw[->, line width=0.03cm, dotted] (x) -- (13,2);
%  \draw[->, line width=0.03cm] (z) to [out=220, in=320] (y);
\end{tikzpicture}  \hspace{1cm}
\begin{tikzpicture}[scale=0.3]
  \node at (7,-1) {(c)};
  \node (y) at (5,1) [vertexX]{$y$};
  \node (x) at (9,1) [vertexX]{$x$}; 

  \node (z) at (5,5) [vertexX]{$z$};
  \node (w) at (9,5) [vertexX]{$w$};

%  \draw[->, line width=0.03cm, dotted] (x) -- (1,2);
  \draw[->, line width=0.03cm] (z) -- (y);
  \draw[->, line width=0.03cm] (w) -- (y);

  \draw[->, line width=0.03cm] (y) -- (x);
  \draw[->, line width=0.03cm] (x) -- (z);
  \draw[->, line width=0.03cm] (x) -- (w);
%  \draw[->, line width=0.03cm, dotted] (x) -- (13,2);

%  \draw[->, line width=0.03cm] (z) to [out=220, in=320] (y);
\end{tikzpicture}  
\caption{The subgraph in (a) is used in Case F.3.a. and the subgraph in (b) is used in Case F.3.b. and the subgraph in (c) is used in Case F.3.c.}
\label{F3} 
\end{center}
\end{figure}
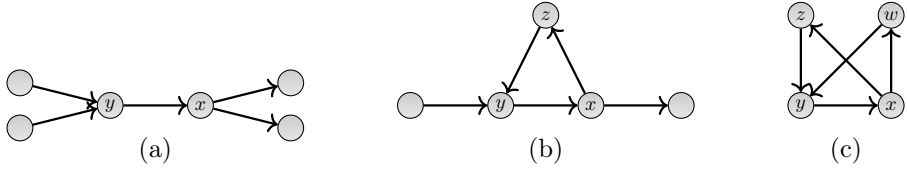

 \item[Case F.3.a. $d^+(y)=1$ and $d^-(y)=2$ and $d^+(x)=2$ and $|N^+(x) \cap N^-(y)| = 0$:] In this case let $D^*$ be obtained from $D$ by contracting the arc $yx$ into a vertex $w$. 
Let $n^*$ be the order of $D^*$ and let $m^*$ be the size of $D^*$.
Let $F^*$ be a minimum feedback vertex set in $D^*$. If $w \in F^*$ then replace it with $x$ (or $y$). Now $F^*$ is a feedback vertex set in $D$, implying that $\fvs(D) \leq \fvs(D^*)$.
As $h(D^*) \leq 1$ 
we obtain the following contradiction, by induction.

\[
\begin{array}{rcl} \vspace{0.1cm}
\fvs(D) & \leq & \fvs(D^*)  \\ \vspace{0.1cm}
       & \leq & \frac{n^*+m^*+h(D^*)}{7}  \\ \vspace{0.1cm}
       & \leq & \frac{(n-1)+(m-1) +1}{7}  \\ 
       &  <   & \frac{n+m}{7} \\
\end{array}
\]

 \item[Case F.3.b. $d^+(y)=1$ and $d^-(y)=2$ and $d^+(x)=2$ and $|N^+(x) \cap N^-(y)| = 1$:] Define $z$ such that $N^+(x) \cap N^-(y)=\{z\}$ and let $W =(N^+(x) \cup N^-(y)) \setminus \{z\}$.
Note that $|W|=2$. Let $C_1,C_2,\ldots,C_l$ be the strong components of $D'$ and without loss of generality assume that $z \in V(C_1)$.
First consider the case when $W \setminus V(C_1) \not= \emptyset$ and let $w \in W \setminus V(C_1)$. By Lemma~\ref{Hlemma}(a-b) we note that there exists a feedback vertex set, $F'$, of $D'$
such that $\{w,z\} \subseteq F'$ and $|F'|=(n'+m'+1)/7$. Note that $F'$ is a feedback vertex set in $D$ as there is only one arc between $\{x,y\}$ and $V(D) \setminus (F' \cup \{x,y\})$.
Therefore we obtain the following contradiction.

\[
\fvs(D) \leq |F'|=\frac{n'+m'+1}{7} = \frac{(n-2)+(m-5)+1}{7} \leq \frac{n+m}{7} 
\]

So we may assume that $W \setminus V(C_1) = \emptyset$. If $C_1=H_2$ then we note that there exists a feedback vertex set, $F_1$, in $C_1$ containing the vertex $z$ (that is, the vertex of degree $2$ in $H_2$)
and one of the vertices in $W$ (that is, the two vertices of degree $3$ in $H_2$). Also we note that $D'=C_1$ and, analogously to above, $F_1$ is a feedback vertex set in $D$. So $\fvs(D)=2$ and $n=7$ and 
$m=13$ so $D$ is not a counter example to the theorem, a contradiction. So, $C_1 \not= H_2$.

 Since $H_3$ has no vertices with degree less than three and $d_{D'}(z) \le 2$, we have that $C_1 \neq H_3$. Therefore, $C_1 = H_1$.

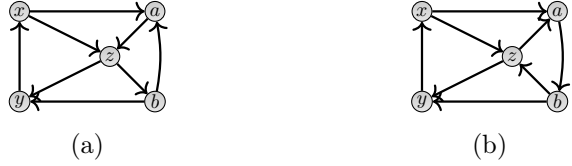
\begin{figure}[htb]
\begin{center}
\tikzstyle{vertexX}=[circle, draw, top color=black!20, bottom color=black!10, minimum size=10pt, scale=0.75, inner sep=0.8pt]
\tikzstyle{vertexY}=[circle, draw, top color=black!90, bottom color=black!70, minimum size=10pt, scale=0.75, inner sep=0.8pt]

\begin{tikzpicture}[scale=0.3]
  \node at (4,-1) {(a)};
  \node (x) at (1,5) [vertexX]{$x$};
  \node (y) at (1,1) [vertexX]{$y$};

  \node (x1) at (7,5) [vertexX]{$a$};
  \node (y1) at (5,3) [vertexX]{$z$};
  \node (z1) at (7,1) [vertexX]{$b$};

  \draw[->, line width=0.03cm] (y) -- (x);

  \draw[->, line width=0.03cm] (x1) -- (y1);
  \draw[->, line width=0.03cm] (y1) -- (z1);
  \draw[->, line width=0.03cm] (z1) to [out=80, in=280] (x1);

  \draw[->, line width=0.03cm] (x) -- (x1);
  \draw[->, line width=0.03cm] (x) -- (y1);
  \draw[->, line width=0.03cm] (y1) -- (y);
  \draw[->, line width=0.03cm] (z1) -- (y);
\end{tikzpicture} \hspace{3cm}
\begin{tikzpicture}[scale=0.3]
  \node at (4,-1) {(b)};
  \node (x) at (1,5) [vertexX]{$x$};
  \node (y) at (1,1) [vertexX]{$y$};

  \node (x1) at (7,5) [vertexX]{$a$};
  \node (y1) at (5,3) [vertexX]{$z$};
  \node (z1) at (7,1) [vertexX]{$b$};

  \draw[->, line width=0.03cm] (y) -- (x);

  \draw[->, line width=0.03cm] (y1) -- (x1);
  \draw[->, line width=0.03cm] (z1) -- (y1);
  \draw[->, line width=0.03cm] (x1) to [out=280, in=80] (z1);

  \draw[->, line width=0.03cm] (x) -- (x1);
  \draw[->, line width=0.03cm] (x) -- (y1);
  \draw[->, line width=0.03cm] (y1) -- (y);
  \draw[->, line width=0.03cm] (z1) -- (y);
\end{tikzpicture}
\caption{The digraph $D[V(C_1) \cup \{x,y\}]$ either looks like (a) or (b) in Case F.3.b, when $C_1=H_1$.} 
\label{Df3b}
\end{center}
\end{figure}

The digraph $D[V(C_1) \cup \{x,y\}]$ now looks either like Figure~\ref{Df3b}(a) or Figure~\ref{Df3b}(b).
If $D[V(C_1) \cup \{x,y\}]$ looks like Figure~\ref{Df3b}(a) then we note that $\{z\}$ is a feedback vertex set in $D[V(C_1) \cup \{x,y\}]$.
Therefore any feedback vertex set of $D'$ containing the vertex $z$ will be a feedback vertex set in $D$ (as $C_1$ is a strong component in $D'$).
So, $\fvs(D) \leq \fvs(D')$ by Lemma~\ref{Hlemma}(b), and we obtain the following contradiction by Lemma~\ref{Hlemma}(a).

\[
\fvs(D) \leq \fvs(D') = \frac{n'+m'+1}{7} \leq \frac{(n-2)+(m-5)+1}{7} < \frac{n+m}{7}
\]

So, we may assume that $D[V(C_1) \cup \{x,y\}]$ looks like Figure~\ref{Df3b}(b). However, this digraph is isomorphic to $H_3$ so we note that $D$ is actually
a digraph in ${\cal H}$, a contradiction to Claim~A. This completes Case F.3.b.

 \item[Case F.3.c. $d^+(y)=1$ and $d^-(y)=2$ and $d^+(x)=2$ and $|N^+(x) \cap N^-(y)| = 2$:] 
 Define $z$ and $w$ such that $N^+(x) \cap N^-(y)=\{z,w\}$.
Let $C_1,C_2,\ldots,C_l$ be the strong components of $D'$ and without loss of generality assume that $z \in V(C_1)$.
First consider the case when $w \not\in V(C_1)$. By Lemma~\ref{Hlemma}(a-b) we note that there exists a feedback vertex set, $F'$, of $D'$
such that $\{w,z\} \subseteq F'$ and $|F'|=(n'+m'+1)/7$. Note that $F'$ is a feedback vertex set in $D$,
which implies that we obtain the following contradiction.

\[
\fvs(D) \leq |F'|=\frac{n'+m'+1}{7} = \frac{(n-2)+(m-5)+1}{7} \leq \frac{n+m}{7} 
\]

So, $\{w,z\} \subseteq V(C_1)$. Therefore $C_1$ contains at least two vertices of degree two, which implies that $C_1=H_1$. 
This implies that $D[V(C_1) \cup \{x,y\}]$ is isomorphic to $H_2$ and therefore $D$  is actually
a digraph in ${\cal H}$, a contradiction to Claim~A. This completes Case F.3.c.

 \item[Case F.4. The remaining cases:] The remaining cases are when $d(x)=3$ and $d(y)=3$ and $d^+(y)=2$. Let $w$ be the unique in-neighbour of $y$ 
in $D$. By Lemma~\ref{Hlemma}(a-b)  we note that there exists a feedback vertex set, $F'$, of $D'$
such that $w \in F'$ and $|F'|=(n'+m'+1)/7$. Note that $F'$ is a feedback vertex set in $D$, 
which implies the following, a contradiction. This completes Case F.4.

\[
\fvs(D) \leq |F'|=\frac{n'+m'+1}{7} = \frac{(n-2)+(m-5)+1}{7} \leq \frac{n+m}{7}
\]

\end{description}

\2

{\bf Claim G} {\em Let $X,Y \subseteq V(D)$ and call $(X,Y)$ a reducible pair, if $X \cap Y = \emptyset$ and for any feedback vertex set, $F'$, in $H-(X \cup Y)$,
$F' \cup X$ is a feedback vertex set in $H$. The following two statemsnts now hold

\begin{description}
\item[(a):] There is no reducible pair, $(X,Y)$, in $D$ where $|X|=2$ and $|Y|=3$.
\item[(b):] There is no reducible pair, $(X,Y)$, in $D$ where $|X|=|Y|=2$.
\item[(c):] There is no reducible pair, $(X,Y)$, in $D$ where $|X|=4$ and $|Y|=5$ and $V(D) \not= X \cup Y$.
\item[(d):] There is no reducible pair, $(X,Y)$, in $D$ where $|X|=4$ and $|Y|\geq 6$.
\end{description}}

\2

{\bf Proof of Claim~G.} For the sake of contradiction assume that we have a reducible pair, $(X,Y)$, in $D$ where $|X|=2$ and $|Y|=3$.
Let $T = D[X \cup Y]$ and let $D'=D-V(T)$ and let $F'$ be a minimum feedback vertex set in $D'$.
Let $n'$ be the order of $D'$ and let $m'$ be the size of $D'$.
By the definition of a reducible pair we note that $F=F' \cup X$ is a feedback vertex set in $D$.
Let $q$ denote the number of ars between $V(T)$ and $V(D) \setminus V(T)$ in $D$. 

If $q=0$ then $D$ (and $T$) is a $2$-regular tournament (by Claim~A and Claim~F) of order $5$. As $X$ is a feedback vertex set in $D$ we note that 
$\fvs(D) \leq 2 < (5+10)/7 = (n+m)/7$, a contradiction to $D$ being a counter example. So $q>0$, which implies that $D'$ is non-empty.

Note that $|A(T)|=10-q/2$, which implies that $n'=n-5$ and $m'=m-(10-q/2)-q = m-10-q/2$.
As $D$ is $2$-regular we note that any connected component of $D'$ that belongs to ${\cal H}$ must have at least $4$ arcs between it and $V(T)$ in $D$ 
(by  Lemma~\ref{Hlemma}(c-d)). Therefore $h(D') \leq q/4$. 
 We now obtain the following contradiction, by induction, thereby proving part (a) of Claim~G.

\[
\begin{array}{rcl} \vspace{0.1cm}
\fvs(D) & \leq & \fvs(D')+|X|  \\ \vspace{0.1cm}
       & \leq & \frac{n'+m'+h(D')}{7} + 2  \\ \vspace{0.1cm}
       & \leq & \frac{(n-5)+(m-10-q/2) +q/4}{7} +2 \\
       & < & \frac{n+m}{7} \\
\end{array}
\]

We now, for the sake of contradiction, assume that we have a reducible pair, $(X,Y)$, in $D$ where $|X|=2$ and $|Y|=2$.
As before, let $T = D[X \cup Y]$ and let $D'=D-V(T)$ and let $F'$ be a minimum feedback vertex set in $D'$
and let $n'$ be the order of $D'$ and let $m'$ be the size of $D'$.
Also, as before, let $q$ denote the number of ars between $V(T)$ and $V(D) \setminus V(T)$ in $D$.

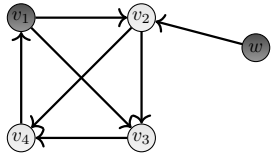
\begin{figure}[htb]
\begin{center}
\tikzstyle{vertexX}=[circle, draw, top color=black!10, bottom color=black!5, minimum size=14pt, scale=0.75, inner sep=0.8pt]
\tikzstyle{vertexY}=[circle, draw, top color=black!70, bottom color=black!40, minimum size=14pt, scale=0.75, inner sep=0.8pt]
\begin{tikzpicture}[scale=0.4]
%  \node at (4,-1) {$H_3$};
%  \draw[->, line width=0.06cm] (z) to [out=210, in=330] (x);

  \node (v1) at (1,5) [vertexY]{$v_1$};
  \node (v2) at (5,5) [vertexX]{$v_2$};
  \node (v3) at (5,1) [vertexX]{$v_3$};
  \node (v4) at (1,1) [vertexX]{$v_4$};

  \node (w) at (8.8,4) [vertexY]{$w$};

  \draw[->, line width=0.03cm] (v1) -- (v2);
  \draw[->, line width=0.03cm] (v2) -- (v3);
  \draw[->, line width=0.03cm] (v3) -- (v4);
  \draw[->, line width=0.03cm] (v4) -- (v1);

  \draw[->, line width=0.03cm] (v1) -- (v3);
  \draw[->, line width=0.03cm] (v2) -- (v4);
  \draw[->, line width=0.03cm] (w) -- (v2);
\end{tikzpicture}
\caption{A subdigraph of $D$ when $q=4$. The dark vertices are the vertices in $X^*$.}
\label{GT}
\end{center}
\end{figure}

We first consider the case when $q \leq 4$. In this case we must have that $T$ is a tournament and $q=4$. 
As the minimum in-degree and the minimum out-degree in $T$ is one (as $D$ was $2$-regular) we note that 
$T$ is a strong tournament.  Therefore it contains a hamilton cycle, say $v_1 v_2 v_3 v_4 v_1$. 
Without loss of generality we may assume that $v_1 v_3, v_2 v_4 \in A(T)$ (there is only one strong tournament on $4$ vertices). 
define $w$ such that $N_D^-(v_2)=\{v_1,w\}$ and let $X^*=\{v_1,w\}$ and let $Y^*=\{v_2,v_3,v_4\}$ (see Figure~\ref{GT}).

Let $F^*$ be a minimum feedback vertex set in $D^*=D - (X^* \cup Y^*)$ and note that $F^* \cup X^*$ is a 
minimum feedback vertex set in $D$ (as both in-neighbours of $v_2$ lie in $X^*$ and both in-neighbours of $v_3$ lie in $X^* \cup \{v_2\}$, 
both in-neighbours of $v_4$ lie in $X^* \cup \{v_2,v_3\}$ and no cycle in $D-F^*-X^*$ can contain vertices from $Y^*$). 
Therefore we have a reducible pair  $(X^*,Y^*)$ in $D$ where $|X^*|=2$ and $|Y^*|=3$, contradicting part~(a).
We may therefore assume that $q\geq 5$.

As above, we note that any connected component of $D'$ that belongs to ${\cal H}$ must have at least $4$ arcs between it and $V(T)$ in $D$
(by  Lemma~\ref{Hlemma}(c-d)). Therefore $h(D') \leq q/4$.
Note that $|A(T)|=8-q/2$, which implies that $n'=n-4$ and $m'=m-(8-q/2)-q = m-8-q/2$.
 We now obtain the following, by induction.

\[
\begin{array}{rcl} \vspace{0.1cm}
\fvs(D) & \leq & \fvs(D')+|X|  \\ \vspace{0.1cm}
       & \leq & \frac{n'+m'+h(D')}{7} + 2  \\ \vspace{0.1cm}
       & \leq & \frac{(n-4)+(m-8-q/2) +q/4}{7} + 2  \\
       & \leq & \frac{n+m-q/4+2}{7} \\
\end{array}
\]

As $q \geq 5$ and $n'$, $m'$ and $h(D')$ are all integers we note that the above implies that $\fvs(D) \leq \frac{n+m}{7}$, a contradiction, thereby proving part (b) of Claim~G.

We now, for the sake of contradiction, assume that we have a reducible pair, $(X,Y)$, in $D$ where $|X|=4$ and $|Y|=5$ and $V(D) \not= X \cup Y$.
As before, let $T = D[X \cup Y]$ and let $D'=D-V(T)$ and let $F'$ be a minimum feedback vertex set in $D'$
and let $n'$ be the order of $D'$ and let $m'$ be the size of $D'$.
Also, as before, let $q$ denote the number of ars between $V(T)$ and $V(D) \setminus V(T)$ in $D$. 
By Claim~A, and the fact that $V(D) \not= X \cup Y$  we have $q \geq 1$.

As above, we note that any connected component of $D'$ that belongs to ${\cal H}$ must have at least $4$ arcs between it and $V(T)$ in $D$
(by  Lemma~\ref{Hlemma}(c-d)). Therefore $h(D') \leq q/4$.
Note that $|A(T)|=18-q/2$, which implies that $n'=n-9$ and $m'=m-(18-q/2)-q = m-18-q/2$.
 We now obtain the following, by induction.

\[
\begin{array}{rcl} \vspace{0.1cm} 
\fvs(D) & \leq & \fvs(D')+|X|  \\ \vspace{0.1cm}
       & \leq & \frac{n'+m'+h(D')}{7} + 4  \\ \vspace{0.1cm}
       & \leq & \frac{(n-9)+(m-18-q/2) +q/4}{7} + 4  \\
       & \leq & \frac{n+m-q/4+1}{7} \\
\end{array}
\]

As $q \geq 1$ and $n'$, $m'$ and $h(D')$ are all integers we note that the above implies that $\fvs(D) \leq \frac{n+m}{7}$, a contradiction, thereby proving part~(c) of Claim~G.

We now, for the sake of contradiction, assume that we have a reducible pair, $(X,Y)$, in $D$ where $|X|=4$ and $|Y| \geq 6$.
Let $y \in Y$ be arbitrary and let $X'=X$ and $Y'=Y \setminus \{y\}$.
Let $F'$ be any feedback vertex set in $D-(X' \cup Y')$. 
If $y \not\in F'$ then $F'$ is a feedback vertex set in $D-(X \cup Y)$ and therefore $X \cup F'$ is a feedback vertex set in $D$ (as $(X,Y)$ is a reducible pair).
And if $y \in F'$, then  $F' \setminus \{y\}$ is a feedback vertex set in $D-(X \cup Y)$ and therefore $X \cup F'$ is again a feedback vertex set in $D$ (even $X \cup F' \setminus \{y\}$ is
a feedback vertex set).
So $(X',Y')$ is a reducible pair contradicting part~(c) and thereby proving part~(d) of Claim~G.

\2

% {\bf Claim H} {\em $D$ does not contain an induced subgraph isomorphic to $H_2$ or $H_3$.}

% \2
       
% {\bf Proof of Claim~H.} Assume, for the sake of contradiction, that $H$ is an induced subgraph of $D$ that is isomorphic to $H_2$. 
% Let $V(H)=\{x,y,z,a,b\}$ and $A(H)=\{xy,yz,zx,ab,bx,xa,bz,za\}$ (see Figure~\ref{H-example}) and let $X_2=\{a,x\}$ and $Y_2=\{b,z\}$.
% Note that $(X_2,Y_2)$ is a reducible pair in $D$ (as both out-neighbours of $z$ lie in $X_2$ and both out-neighbours of $b$ lie in $X_2 \cup \{z\}$),
% contradicting Claim~G part (b).

% Now, for the sake of contradiction, assume that $H$ is an induced subgraph of $D$ that is isomorphic to $H_3$.
% Let $V(H)=\{x,y,z,a,b\}$ and $A(H)=\{xy,yz,zx,ab,bx,ya,bz,za\}$
% (see Figure~\ref{H-example}) and let $X_3=\{a,x\}$ and $Y_2=\{b,z,y\}$.
% Note that $(X_3,Y_3)$ is a reducible pair in $D$ (as both out-neighbours of $z$ lie in $X_2$ and both out-neighbours of $b$ lie in $X_2 \cup \{z\}$ and
% both out-neighbours of $y$ lie in $X_2 \cup \{b,z\}$),
% contradicting Claim~G part~(a). This completes the proof of Claim~H.

\2

{\bf Claim H} {\em $D$ does not contain a transitive triangle (that is, an acyclic orientation of a $3$-cycle).}

\2

%zzzzzzzP

\begin{figure}[htb]
\begin{center}
\tikzstyle{vertexX}=[circle, draw, top color=black!10, bottom color=black!5, minimum size=14pt, scale=0.75, inner sep=0.8pt]
\tikzstyle{vertexY}=[circle, draw, top color=black!70, bottom color=black!40, minimum size=14pt, scale=0.75, inner sep=0.8pt]
\begin{tikzpicture}[scale=0.3]
  \node at (5.5,-1) {(a)};

  \node (x) at (1,5) [vertexY]{$x$};
  \node (y) at (1,1) [vertexY]{$y$};
  \node (u) at (5,3) [vertexX]{$u$};
  \node (v) at (10,3) [vertexX]{$v$};

  \draw[->, line width=0.03cm] (y) -- (u);
  \draw[->, line width=0.03cm] (x) -- (u);
  \draw[->, line width=0.03cm] (u) -- (v);
%  \draw[->, line width=0.03cm] (x) -- (v);
  \draw[->, line width=0.03cm] (x) to [out=0, in=135] (v);
\end{tikzpicture} \hspace{2.5cm}
\begin{tikzpicture}[scale=0.3]
  \node at (5.5,-1) {(b)};
  
  \node (x) at (1,3) [vertexY]{$x$};
  \node (u) at (5,5) [vertexX]{$u$};
  \node (v) at (5,1) [vertexX]{$v$};
  \node (y) at (9,3) [vertexY]{$y$};

  \draw[->, line width=0.03cm] (x) -- (u);
  \draw[->, line width=0.03cm] (x) -- (v);
  \draw[->, line width=0.03cm] (u) -- (y);
  \draw[->, line width=0.03cm] (v) -- (y);
\end{tikzpicture} 
\caption{The digraph in (a) is a subdigraph of $D$ in the proof of Claim~H, where the dark vertices are the vertices in $X$. The digraph in (b) is the digraph $Q$ mentioned in  Claim~J.}
\label{ClaimInJ}
\end{center}
\end{figure}
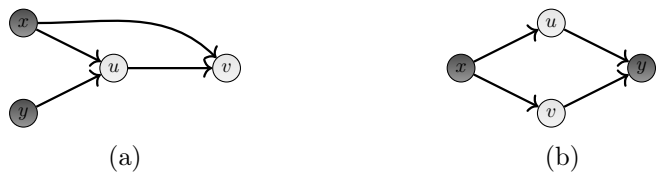

{\bf Proof of Claim~H.}  For the sake of contradiction assume that $xu,uv,xv \in A(D)$. Define $y$ such that $N_D^-(u) = \{x,y\}$
and let $X=\{x,y\}$ and let $Y=\{u,v\}$ (see Figure~\ref{ClaimInJ}(a)). 
Note that $(X,Y)$ is a reducible pair in $D$ (as both in-neighbours of $u$ lie in $X$ and both in-neighbours of $v$ lie in $X \cup \{u\}$),
contradicting Claim~G part~(b). This completes the proof of Claim~H.

\2
{\bf Claim I} {\em $D$ does not contain an induced subgraph isomorphic to $H_2$ or $H_3$.}

\2
       
{\bf Proof of Claim~I.} Assume, for the sake of contradiction, that $H$ is an induced subgraph of $D$ that is isomorphic to $H_2$ or $H_3$. Then $H$ contains a transitive triangle, contradicting Claim~H. This completes the proof of Claim~I.
\2

{\bf Claim J} {\em $D$ does not contain a subgraph isomorphic to $Q$, where $V(Q)=\{x,y,u,v\}$ and $A(Q)=\{xu,uy,xv,vy\}$ (see Figure~\ref{ClaimInJ}(b)).}

\2

{\bf Proof of Claim~J.}  For the sake of contradiction assume $D$ contains a subgraph isomorphic to $Q$ and 
let $X=\{u,v\}$ and let $Y=\{x,y\}$. Note that $(X,Y)$ is a reducible pair in $D$ (as both in-neighbours of $y$ lie in $X$ and both out-neighbours of $v$ lie in $X$),
contradicting Claim~G part~(b). This completes the proof of Claim~J.

\2

{\bf Claim K} {\em There is no reducible pair, $(X,Y)$, in $D$ where $|X|=3$ and $|Y|=3$ and $|A(D[X \cup Y])| \leq 8$.}

\2

{\bf Proof of Claim~K.} 
For the sake of contradiction assume that we have a reducible pair, $(X,Y)$, in $D$ where $|X|=|Y|=3$ and $|A(D[X \cup Y])| \leq 8$.
Let $T = D[X \cup Y]$ and let $D'=D-V(T)$ and let $F'$ be a minimum feedback vertex set in $D'$
and let $n'$ be the order of $D'$ and let $m'$ be the size of $D'$.
Let $q$ denote the number of ars between $V(T)$ and $V(D) \setminus V(T)$ in $D$.
As $12-q/2 = |A(D[X \cup Y])| \leq 8$ we note that $q \geq 8$.

By Claim~I and Lemma~\ref{Hlemma}(c), we note that any connected component of $D'$ that belongs to ${\cal H}$ must have at least $6$ arcs between it and $V(T)$ in $D$.
Therefore $h(D') \leq q/6$. As $|A(T)|=12-q/2$, we have $n'=n-6$ and $m'=m-(12-q/2)-q = m-12-q/2$.
 We now obtain the following, by induction.

\[
\begin{array}{rcl} \vspace{0.1cm}
\fvs(D) & \leq & \fvs(D')+|X|  \\ \vspace{0.1cm}
       & \leq & \frac{n'+m'+h(D')}{7} + 3  \\ \vspace{0.1cm}
       & \leq & \frac{(n-6)+(m-12-q/2) +q/6}{7} + 3  \\
       & \leq & \frac{n+m-q/3+3}{7} \\
\end{array}
\]

As $q \geq 8$ and $n'$, $m'$ and $h(D')$ are all integers we note that the above implies that $\fvs(D) \leq \frac{n+m}{7}$, a contradiction, thereby proving Claim~K.

\2

{\bf Claim L} {\em Every arc in $D$ belongs to exactly one $3$-cycle.}

\2

{\bf Proof of Claim~L.} First assume that some arc $xy \in A(D)$ belongs to two $3$-cycles $xyux$ and $xywx$ where $u \not=w$. 
and let $X=\{u,w\}$ and let $Y=\{x,y\}$. Note that $(X,Y)$ is a reducible pair in $D$ (as both in-neighbours of $x$ lie in $X$ and both out-neighbours of $y$ lie in $X$),
contradicting Claim~G part~(b). Therefore any arc in $D$ belongs to at most one $3$-cycle.

\begin{figure}[htb]
\begin{center}
\tikzstyle{vertexX}=[circle, draw, top color=black!10, bottom color=black!5, minimum size=14pt, scale=0.75, inner sep=0.8pt]
\tikzstyle{vertexY}=[circle, draw, top color=black!70, bottom color=black!40, minimum size=14pt, scale=0.75, inner sep=0.8pt]
\begin{tikzpicture}[scale=0.3]
  \node at (5.5,-1) {(a)};

  \node (u1) at (1,6) [vertexY]{$u_1$};
  \node (u2) at (5,6) [vertexX]{$u_2$};
  \node (u3) at (9,6) [vertexY]{$u_3$};

  \node (x) at (3,1) [vertexX]{$x$};
  \node (y) at (9,1) [vertexX]{$y$};
  \node (w) at (9,8) [vertexY]{$w$};

  \draw[->, line width=0.03cm] (x) -- (y);
  \draw[->, line width=0.03cm] (u1) -- (x);
  \draw[->, line width=0.03cm] (u2) -- (x);
  \draw[->, line width=0.03cm] (u3) -- (y);

  \draw[->, line width=0.03cm] (u3) -- (u2);
  \draw[->, line width=0.03cm] (w) -- (u2);

%  \draw[->, line width=0.03cm] (x) -- (v);
%  \draw[->, line width=0.03cm] (x) to [out=0, in=135] (v);
\end{tikzpicture} \hspace{2.5cm}
\begin{tikzpicture}[scale=0.3]
  \node at (5.5,-1) {(b)};

  \node (u1) at (1,6) [vertexX]{$u_1$};
  \node (u2) at (5,6) [vertexX]{$u_2$};
  \node (u3) at (9,6) [vertexX]{$u_3$};

  \node (x) at (3,1) [vertexX]{$x$};
  \node (y) at (9,1) [vertexX]{$y$};

  \draw[->, line width=0.03cm, dotted] (u1) -- (2,8);
  \draw[->, line width=0.03cm, dotted] (1,8) -- (u1);
  \draw[->, line width=0.03cm, dotted] (0,8) -- (u1);

  \draw[->, line width=0.03cm, dotted] (u2) -- (6,8);
  \draw[->, line width=0.03cm, dotted] (5,8) -- (u2);
  \draw[->, line width=0.03cm, dotted] (4,8) -- (u2);

  \draw[->, line width=0.03cm, dotted] (u3) -- (10,8);
  \draw[->, line width=0.03cm, dotted] (9,8) -- (u3);
  \draw[->, line width=0.03cm, dotted] (8,8) -- (u3);

  \draw[->, line width=0.03cm, dotted] (x) -- (1,1);

  \draw[->, line width=0.03cm, dotted] (y) -- (11,0.5);
  \draw[->, line width=0.03cm, dotted] (y) -- (11,1.5);

  \draw[->, line width=0.03cm] (x) -- (y);
  \draw[->, line width=0.03cm] (u1) -- (x);
  \draw[->, line width=0.03cm] (u2) -- (x);
  \draw[->, line width=0.03cm] (u3) -- (y);

%  \draw[->, line width=0.03cm] (x) -- (v);
%  \draw[->, line width=0.03cm] (x) to [out=0, in=135] (v);
\end{tikzpicture} 
\caption{The digraph in (a) is a subdigraph of $D$ considered in the proof of Claim~L, where the dark vertices are the vertices in $X$. The digraph in (b) is then shown to be an induced subdigraph of $D$ in the proof of Claim~L.}
\label{ClaimLi}
\end{center}
\end{figure}
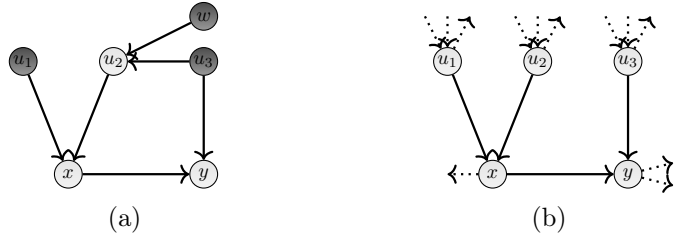

Assume for the sake of contradiction that $xy$ does not belong to any $3$-cycle. And define $u_1,u_2,u_3$ such that $N_D^-(x)=\{u_1,u_2\}$ and $N_D^-(y)=\{x,u_3\}$.
Let $T = D[\{u_1,u_2,u_3,x,y\}]$ and let $D' = D - V(T)$ and let $n'$ be the order of $D'$ and let $m'$ be the size of $D'$.
We first show that $A(T)=\{u_1x,u_2x,xy,u_3y\}$. As $xy$ belongs to no $3$-cycle we have $yu_1,y u_2 \not\in A(T)$. By Claim~H there is no transitive triangle
in $D$ and therefore $u_1y,u_2y,u_1u_2,u_2u_1,u_3x,xu_3 \not\in A(T)$. 
By Claim~J we note that $u_1 u_3, u_2 u_3 \not\in A(T)$. 

For the sake of contradiction assume that $u_3 u_i \in A(T)$ for some $i \in \{1,2\}$.
Without loss of generality assume that $u_3 u_2 \in A(T)$. Let $w$ be defined such that $N^-(u_2)=\{u_3,w\}$ and let
$X=\{w,u_1,u_3\}$ and $Y=\{x,y,u_2\}$ (see Figure~\ref{ClaimLi}(a)). By the above arguments we note that $X \cap Y = \emptyset$ and $|X|=3$ and $|Y|=3$.
Note that $(X,Y)$ is a reducible pair in $D$ (as both in-neighbours of $u_2$ lie in $X$ and both in-neighbours of $x$ lie in $X \cup \{u_2\}$ and
both in-neighbours of $y$ lie in $X \cup \{u_2,x\}$).
We will now show that $|A(D[X \cup Y])| \leq 8$, which will give us a contradiction by Claim~K.

Note that $u_3 u_1 \not\in A(D)$ by Claim~J (as $u_3u_2, u_2x, u_1x \in A(T)$). Therefore $A(T)= \{u_1x,u_2x,xy,u_3y,u_3u_2\}$.
By Claim~H there is no transitive triangle in $D$ and therefore $wu_3,u_3w \not\in A(T)$.
By Claim~H we can also not have both arcs $ xw$ and $yw$ in $T$, which implies that the neighbours of $w$ in $V(T)$ is a proper subset of $\{u_1,u_2,x,y\}$ 
(as $N^-(x)=\{u_1,u_2\}$ and $N^-(y)=\{x,u_3\}$). As $A(T)= \{u_1x,u_2x,xy,u_3y,u_3u_2\}$ and $w$ has at most three arcs incident with vertices in $T$ we note that
$|A(D[X \cup Y])| \leq 8$. Therefore we obtain a contradiction by Claim~K.
This implies that $u_3 u_1, u_3 u_2 \not\in A(T)$. This implies that $A(T)=\{u_1x,u_2x,xy,u_3y\}$, as desired (see Figure~\ref{ClaimLi}(b)).

First consider the case when $h(D')=0$. In this case we note that $n'=n-5$ and $m'=m-16$. 
As adding $\{u_1,u_2,u_3\}$ to any feedback vertex set in $D'$ gives us a feedback vertex set in $D$ we note that the following holds.

\[
\fvs(D) \leq \fvs(D')+3 \leq  \frac{n'+m'+h(D')}{7} + 3  =  \frac{n+m}{7} 
\]

Therefore we must have $h(D') \geq 1$. Let $R$ be a component in $D'$ which belongs to ${\cal H}$. 
As there are twelve arcs between $V(T)$ and $V(D) \setminus V(T)$ we note that $R$ is ether a $3$-cycle or
two $3$-cycles connected by an arc (by Claim~I).

First consider the case when $R$ is two $3$-cycles connected by an arc. In this case there are 10 arcs between $V(T)$ and $V(D) \setminus V(T)$.
This implies that the two in-neighbours, say $a$ and $b$, of $u_r$ belong to $V(R)$ for some $r \in \{1,2,3\}$.
The two in-neighbours, $a$ and $b$, cannot belong to the same $3$-cycle by Claim~H. So $\{a,b\}$ is a feedback vertex set of $R$.
Let $X=\{a,b,u_1,u_2,u_3\} \setminus \{u_r\}$ and let $Y=(V(T) \cup V(R)) \setminus X$ and
note that $(X,Y)$ is a reducible pair in $D$, with $|X|=4$ and $|Y|=7$, a contradiction against Claim~G part~(d).

So $R$ is a $3$-cycle. As there are three arcs from $V(R)$ to $V(T)$ in $D$ and no such arc can enter $\{x,y\}$ (as $d^-(x)=d^-(y)=2$) and
no two such arcs can enter the same vertex in $\{u_1,u_2,u_3\}$ by Claim~H we note that there is an arc from $V(R)$ into every vertex in $\{u_1,u_2,u_3\}$.
As there are also three arcs from $V(T)$ to $V(R)$ in $D$ and by Claim~H at most two of these are adjacent to $\{x,y\}$ we note that 
there exists an arc from $u_j$ to $V(R)$ for some $j \in \{1,2,3\}$. Let $V(R)=\{r_1,r_2,r_3\}$ and let $a$ and $b$ be defined such that
$r_a u_j, u_j r_b \in A(D)$. Note that $a \not= b$ as $D$ is an oriented graph. Without loss of generality assume that $A(R)=\{r_1 r_2, r_2 r_3, r_3 r_1\}$ and
$a=2$, which implies that $b \in \{1,3\}$. Note that $b \not= 3$ by Claim~H and $b \not= 1$ by Claim~J (as if $b=1$ then $r_2 u_j, u_j r_1, r_2 r_3, r_3 r_1 \in A(D)$),
a contradiction.  This completes the proof of Claim~L.

Let $G_D=UG(D)$ be the underlying graph of $D$.
\2

{\bf Claim M} {\em For every $u \in V(G_D)$ the set $N(v)$ is an induced matching of size two.
That is $G_D[N(v)]$ is $1$-regular.}

\2

{\bf Proof of Claim~M.} Let $v$ be any vertex in $G_D$ and let $N(v)=\{u_1,u_2,u_3,u_4\}$ and  let $R=G_D[N(v)]$ (ie the graph induced by $\{u_1,u_2,u_3,u_4\}$ in $G_D$).
First assume that som vertex, say $u_1$, in $R$ has degree at least two in $R$. Without loss of generality assume that $u_1 u_2, u_1 u_3 \in E(R)$ and note
that $C_1 = v u_1 u_2 v$ and $C_2 = v u_1 u_3 v$ are two distinct $3$-cycles in $R$. By Claim~H we note that neither $C_1$ or $C_2$ are transitive triples, 
which implies that $vu_1$ belongs to two distinct $3$-cycles, a contradiction to Claim~L. So $\Delta(R) \leq 1$.

Now assume that some vertex in $R$, say $u_1$, is isolated in $R$. Then no $3$-cyle in $D$ contains the arc between $v$ and $u_1$, a contradiction to Claim~L.
This implies that $\delta(R)=\Delta(R)=1$, which completes the proof of Claim~M.

\2

{\bf Claim N} {\em $G_D$ contains no $4$-cycles.}

\2

{\bf Proof of Claim~N.}  For the sake of contradiction let $C=c_1 c_2 c_3 c_4 c_1$ be a $4$-cycle in $G_D$.
Note that $c_2 c_4, c_1 c_3 \not\in E(G_D)$, by Claim~M, so $C$ is an induced $4$-cycle in $G_D$.
Let $D_1 = D - \{c_1,c_3\}$ and note that both $c_2$ and $c_4$ have degree two in $D_1$.
For each $i \in \{2,4\}$ do the following.

\begin{description}
 \item[Case A:] If $d^+_{D_1}(c_i)=0$ or $d^-_{D_1}(c_i)=0$, then delete $c_i$ from $D_1$.
 \item[Case B:] If $d^+_{D_1}(c_i)=d^-_{D_1}(c_i)=1$, then let $u_i c_i v_i$ be a path in $D_1$ and delete $c_i$ from $D_1$ and add the arc $u_i v_i$ to $D_1$.
\end{description}

  Let the resulting digraph be denoted by $D_2$. Note that in Case~B, when we add the arc $u_i v_i$ we must have $u_iv_i, v_i u_i \not\in A(D_1)$, because of 
Claim~M (when considering the set $N(c_i)$).  For the sake of contradiction assume that we are in Case~B for both $i=2$ and $i=4$ and $\{u_2,v_2\}=\{u_4,v_4\}$.
In this case we must have an arc between a vertex in $\{u_2,v_2\}$ and $\{c_1,c_3\}$ by Claim~M. Without loss of generality assume the arc goes between
$u_2$ and $c_1$. But now the edges $c_1 c_2$ and $c_1 c_4$ both belong to the neighbourhood of $u_2$, a contradiction to Claim~M. So, this implies that 
$D_2$ is an oriented graph. Also, it is not difficult to see that the maximum degree in $D_2$ is at most four.

  Assume that we were in Case~B $q$ times above ($0 \leq q \leq 2$). Let $n_2$ be the order of $D_2$ and let $m_2$ be the size of $D_2$. Now $n_2=n-4$ and 
$m_2 = m - 12+ q$. Furthermore, if $F_2$ is a feeback vertex set in $D_2$ then it is not difficult to see that $F_2 \cup \{c_1,c_3\}$ is a feedback vertex
set in $D$, which implies that $\fvs(D) \leq \fvs(D_2)+2$.  If $h(D_2) \leq 2-q$, then we obtain the following contradiction,by induction.

\[
\begin{array}{rcl} \vspace{0.1cm} 
\fvs(D) & \leq & \fvs(D_2)+2  \\ \vspace{0.1cm}
       & \leq & \frac{n_2+m_2+h(D_2)}{7} + 2  \\ \vspace{0.1cm}
       & \leq & \frac{(n-4)+(m-12+q) + 2-q}{7} + 2  \\
       & \leq & \frac{n+m}{7} \\
\end{array}
\]

So, we may assume that $h(D_2) > 2-q$. Note that $\sum_{x \in V(D_2)} (4-d_{D_2}(x)) = 8 - 2q$. 
If $q=0$ then by Lemma~\ref{Hlemma}(c-d) we note that $h(D_2) \leq 2$ (in fact $h(D_2) \leq 1$ by Claim~I), a contradiction.
If $q=1$ then by Lemma~\ref{Hlemma}(c-d) we note that $h(D_2) \leq 1$, a contradiction.
So we must have $q=2$ and $h(D_2) = 1$, which implies that $D_2$ contains a component, $R$, isomorphic to $H_2$ or $H_3$.
Furthermore by the above we can assume that every $4$-cycle in $G_D$ is an oriented $4$-cycle in $D$, as otherwise we 
could have let $C$ be a $4$-cycle in $G_D$, which is not an oriented $4$ cycle in $D$ and named the vertices in $C$, such that $q \leq 1$.
This means that if there is a transitive triangle in $R$ then some arc in $R$ was an arc of the form $u_i v_i$ where $i \in \{2,4\}$ by Claim~H and in fact
both arcs  arcs $u_2 v_2$ and $u_4 v_4$ must belong to $R$ as otherwise we would have a $4$-cycle in $G_D$ which was not an oriented $4$-cycle in $D$.

\begin{figure}[htb]
\begin{center}
\tikzstyle{vertexX}=[circle, draw, top color=black!20, bottom color=black!10, minimum size=10pt, scale=0.75, inner sep=0.8pt]
\begin{tikzpicture}[scale=0.3]
  \node at (4,-1) {$H_2$};
  \node (x) at (1,1) [vertexX]{$x$};
  \node (y) at (4,1) [vertexX]{$y$};
  \node (z) at (7,1) [vertexX]{$z$};

  \node (a) at (1,5) [vertexX]{$a$};
  \node (b) at (7,5) [vertexX]{$b$};
  \draw[->, line width=0.03cm] (x) -- (y);
  \draw[->, line width=0.03cm] (y) -- (z);
  \draw[->, line width=0.06cm] (z) to [out=210, in=330] (x);
  \draw[->, line width=0.03cm] (a) -- (b);
  \draw[->, line width=0.06cm] (b) -- (x);
  \draw[->, line width=0.06cm] (b) -- (z);
  \draw[->, line width=0.06cm] (x) -- (a);
  \draw[->, line width=0.06cm] (z) -- (a);
\end{tikzpicture} \hspace{1cm}
\begin{tikzpicture}[scale=0.3]
  \node at (4,-1) {$H_3$};
  \node (x) at (1,1) [vertexX]{$x$};
  \node (y) at (4,1) [vertexX]{$y$};
  \node (z) at (7,1) [vertexX]{$z$};

  \node (a) at (1,5) [vertexX]{$a$};
  \node (b) at (7,5) [vertexX]{$b$};
  \draw[->, line width=0.03cm] (x) -- (y);
  \draw[->, line width=0.06cm] (y) -- (z);
  \draw[->, line width=0.06cm] (z) to [out=210, in=330] (x);
  \draw[->, line width=0.03cm] (a) -- (b);
  \draw[->, line width=0.06cm] (b) -- (x);
  \draw[->, line width=0.06cm] (b) -- (z);
  \draw[->, line width=0.06cm] (y) -- (a);
  \draw[->, line width=0.06cm] (z) -- (a);
\end{tikzpicture}
\caption{The digraphs $H_2$ and $H_3$, where the thick arcs belong to transitive triangles.}
\label{H-digraphsII}
\end{center}
\end{figure}
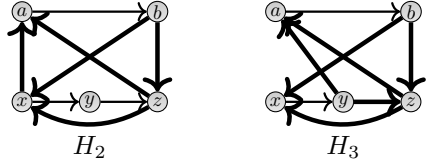

Name the vertices in $R$ as in Figure~\ref{H-digraphsII}, depending on if $R=H_2$ or $R=H_3$.
First assume that $R=H_2$. As both $\{z,x,a\}$ and $\{z,x,b\}$ induce a transitive triangle the arcs $u_2 v_2$ and $u_4 v_4$ both lie within both sets,
which is impossible as the two sets only have one arc in common (the one between $z$ and $x$).
So, $R=H_3$. But now the sets  both $\{z,y,a\}$ and $\{z,x,b\}$ induce a transitive triangles in $R$.
Therefore, the arcs $u_2 v_2$ and $u_4 v_4$ both lie within both sets, a contradiction, as the sets have no arcs in common. This completes the proof of Claim~N.

\2

{\bf Claim O} {\em Any $5$-cycle, $C=c_1 c_2 c_3 c_4 c_5 c_1$, in $G_D$ is a directed cycle in $D$ (this means that 
either $c_1 c_2 c_3 c_4 c_5 c_1$ or $c_1 c_5 c_4 c_3 c_2 c_1$ is a directed cycle in $D$).}

\2

{\bf Proof of Claim~O.}   For the sake of contradiction let $C=c_1 c_2 c_3 c_4 c_5 c_1$ be a $5$-cycle in $G_D$, 
such that $C$ is not a directed $5$-cycle in $D$.
Let $x_i$ be the vertex in $G_D$ such that $\{c_i,c_{i+1},x_i\}$ induces a $3$-cycle in $D$, where all subscripts are taken modulo $5$. Such an $x_i$ exists as
the arc between $c_i$ and $c_{i+1}$ belongs to a unique $3$-cycle in $D$ by Claim~L.
By Claim~N we note that $x_1,x_2,x_3,x_4,x_5$ are distinct vertices. See Figure~\ref{C5pic}.

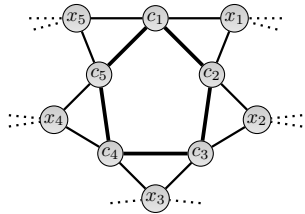
\begin{figure}[htb]
\begin{center}
\tikzstyle{vertexX}=[circle, draw, top color=black!20, bottom color=black!10, minimum size=10pt, scale=0.75, inner sep=0.8pt]
\begin{tikzpicture}[scale=0.3]
%  \node at (4,-3) {$H_2$};
  \node (c1) at (3,7) [vertexX]{$c_1$};
  \node (c2) at (5.5,4.5) [vertexX]{$c_2$};
  \node (c3) at (5,1) [vertexX]{$c_3$};
  \node (c4) at (1,1) [vertexX]{$c_4$};
  \node (c5) at (0.5,4.5) [vertexX]{$c_5$};
  \node (x1) at (6.5,7) [vertexX]{$x_1$};
  \node (x2) at (7.5,2.5) [vertexX]{$x_2$};
  \node (x3) at (3,-1) [vertexX]{$x_3$};
  \node (x4) at (-1.5,2.5) [vertexX]{$x_4$};
  \node (x5) at (-0.5,7) [vertexX]{$x_5$};
  \draw[line width=0.05cm] (c1) -- (c2);
  \draw[line width=0.05cm] (c2) -- (c3);
  \draw[line width=0.05cm] (c3) -- (c4);
  \draw[line width=0.05cm] (c4) -- (c5);
  \draw[line width=0.05cm] (c5) -- (c1);

  \draw[line width=0.03cm] (c1) -- (x1);
  \draw[line width=0.03cm] (c2) -- (x1);

  \draw[line width=0.03cm] (c2) -- (x2);
  \draw[line width=0.03cm] (c3) -- (x2);

  \draw[line width=0.03cm] (c3) -- (x3);
  \draw[line width=0.03cm] (c4) -- (x3);

  \draw[line width=0.03cm] (c4) -- (x4);
  \draw[line width=0.03cm] (c5) -- (x4);

  \draw[line width=0.03cm] (c5) -- (x5);
  \draw[line width=0.03cm] (c1) -- (x5);

  \draw[line width=0.03cm, dotted] (8.6,7) -- (x1);
  \draw[line width=0.03cm, dotted] (8.4,6.5) -- (x1);
  \draw[line width=0.03cm, dotted] (9.5,2.7) -- (x2);
  \draw[line width=0.03cm, dotted] (9.5,2.2) -- (x2);
  \draw[line width=0.03cm, dotted] (1,-1.2) -- (x3);
  \draw[line width=0.03cm, dotted] (5,-1.2) -- (x3);
  \draw[line width=0.03cm, dotted] (-3.5,2.7) -- (x4);
  \draw[line width=0.03cm, dotted] (-3.5,2.2) -- (x4);
  \draw[line width=0.03cm, dotted] (-2.6,7) -- (x5);
  \draw[line width=0.03cm, dotted] (-2.4,6.5) -- (x5);
\end{tikzpicture} 
\caption{The subgraph of $G_D$ used in the proof of Claim~O.}
\label{C5pic}
\end{center}
\end{figure}
 
Let $X=\{x_1,x_2,x_3,x_4,x_5\}$ and $Y=V(C)$.
By Claim~N we also note that $X$ is an independent set and also by Claim~N we note that $C$ is an induced cycle in $G_D$, as any chord of $C$ would give us a $4$-cycle.
As $C$ is not a directed $5$-cycle in $D$  we note that $(X,Y)$ is a reducible pair in $D$.
Let $D'=D - (X \cup Y)$  and let $n'$ be the order of $D'$ and let $m'$ be the size of $D'$.
Note that $m'=m-25$ and $n'=n-10$. 
If $h(D') =0$, then we obtain the following contradiction,by induction.

\[
\begin{array}{rcl} \vspace{0.1cm}
\fvs(D) & \leq & \fvs(D')+ |X| \\ \vspace{0.1cm}
       & \leq & \frac{n'+m'+h(D')}{7} + 5  \\ \vspace{0.1cm}
       & = & \frac{(n-10)+(m-25)}{7} + 5  \\
       & = & \frac{n+m}{7} \\
\end{array}
\]

So, $h(D') > 0$. Let $H'$ be a component of $D'$ such that $H' \in {\cal H}$.
By Claim~I we note that $H'$ contains no copy of $H_2$ or $H_3$. 
So by the construction of ${\cal H}$, there must be a $3$-cycle $y_1 y_2 y_3 y_1$ in $H'$ such that $y_1$ has degree two in $H'$.
As $D$ is $2$-regular this implies that there exists some $x_i, x_j \in X$ such that $N_D^+(y_1)=\{y_2,x_j\}$ and $N_D^-(y_1)=\{y_3,x_i\}$.
By Claim~M we note that $x_j x_i \in A(D)$, contradicting the fact that $X$ is an independent set in $D$. This completes the proof of Claim~O.

\2

{\bf Claim P} {\em We obtain a contradiction.}

\2

{\bf Proof of Claim~P.} Let $x_2$ be any vertex in $V(D)$ and let $N_D^+(x_2)=\{y_1,y_2\}$. Define $x_1$ and $x_3$, such that 
$N_D^-(y_1)=\{x_1,x_2\}$ and $N_D^-(y_2)=\{x_2,x_3\}$. Note that $x_1 \not= x_3$ by Claim~N.
Let $X=\{x_1,x_2,x_3\}$ and $Y=\{y_1,y_2\}$. By Claim~H and Claim~O we note that $X$ and $Y$ are independent sets in $D$.
By Claim~N we note that $A(D[X \cup Y])=\{x_1 y_1, x_2 y_1, x_2 y_2, x_3 y_2\}$ (see Figure~\ref{ClaimPi}).

\begin{figure}[htb]
\begin{center}
\tikzstyle{vertexX}=[circle, draw, top color=black!20, bottom color=black!10, minimum size=10pt, scale=0.75, inner sep=0.8pt]
\begin{tikzpicture}[scale=0.3]
%  \node at (4,-3) {$H_2$};
  \node (x1) at (1,9) [vertexX]{$x_1$};
  \node (x2) at (1,5) [vertexX]{$x_2$};
  \node (x3) at (1,1) [vertexX]{$x_3$};

  \node (y1) at (6,8) [vertexX]{$y_1$};
  \node (y2) at (6,2) [vertexX]{$y_2$};

  \draw[->, line width=0.03cm] (x1) -- (y1);
  \draw[->, line width=0.03cm] (x2) -- (y1);
  \draw[->, line width=0.03cm] (x2) -- (y2);
  \draw[->, line width=0.03cm] (x3) -- (y2);

  \draw[->, line width=0.03cm, dotted] (-2,10) -- (x1);
  \draw[->, line width=0.03cm, dotted] (-2,9) -- (x1);
  \draw[->, line width=0.03cm, dotted] (x1) -- (-2,8);

  \draw[->, line width=0.03cm, dotted] (-2,5.5) -- (x2);
  \draw[->, line width=0.03cm, dotted] (x2) -- (-2,4.5);

  \draw[->, line width=0.03cm, dotted] (-2,2) -- (x3);
  \draw[->, line width=0.03cm, dotted] (-2,1) -- (x3);
  \draw[->, line width=0.03cm, dotted] (x3) -- (-2,0);

  \draw[->, line width=0.03cm, dotted] (y1) -- (9,8.5);
  \draw[->, line width=0.03cm, dotted] (y1) -- (9,7.5);

  \draw[->, line width=0.03cm, dotted] (y2) -- (9,2.5);
  \draw[->, line width=0.03cm, dotted] (y2) -- (9,1.5);
\end{tikzpicture}
\caption{a subgraph of $D$ used in the proof of Claim~P.}
\label{ClaimPi}
\end{center}
\end{figure}
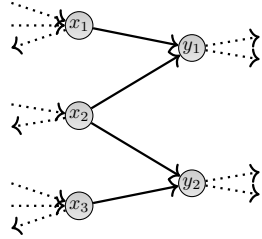

Let $D'=D - (X \cup Y)$  and let $n'$ be the order of $D'$ and let $m'$ be the size of $D'$.
Note that $m'=m-16$ and $n'=n-5$. 
If $h(D') =0$, then we obtain the following contradiction,by induction.
  
\[
\begin{array}{rcl} \vspace{0.1cm}
\fvs(D) & \leq & \fvs(D')+ |X| \\ \vspace{0.1cm}
       & \leq & \frac{n'+m'+h(D')}{7} + 3  \\ \vspace{0.1cm}
       & = & \frac{(n-5)+(m-16)}{7} + 3  \\
       & = & \frac{n+m}{7} \\
\end{array}
\]
  
So, $h(D') > 0$. Let $H'$ be a component of $D'$ such that $H' \in {\cal H}$.
By Claim~I we note that $H'$ contains no copy of $H_2$ or $H_3$.
Let $Z=\{z_1,z_2, \ldots, z_l\}$ contain all vertices in $H'$ of degree two in $H'$.
By the construction of ${\cal H}$ we note that $|Z| \geq 3$ and every vertex in $Z$ (and in $H'$) belongs to a $3$-cycle in $H'$.
For each $i=1,2,\ldots,l$ there exist vertices $s_i,t_i \in V(D) \setminus V(H')$, such that $z_i s_i t_i z_i$ is a $3$-cycle in $D$ (by Claim~H and Claim~M and the fact that
$z_i$ belongs to a $3$-cycle within $H'$). As $s_i,t_i \in X \cup Y$ and $s_i t_i \in A(D)$ we note that $t_i \in Y$ (and $s_i \in X$).
This implies that $|Z| \leq 4$.
Furthermore, at least two of the vertices in $t_1,t_2,\ldots,t_l$ are identical as $l \geq 3$ and $|Y|=2$. Without loss of generality assume that $t_1=t_2$.
As $|Z| \leq 4$ we not that $|V(H')|=3$ or $|V(H')|=6$ and the distance in $G_D$ between any two vertices in $H'$ is at most $3$.
Now a shortest path between $z_1$ and $z_2$ together with the arcs $t_1 z_1$ and $t_2 z_2$ ($t_1=t_2$) form a $3$-cycle or a $4$-cycle or a $5$-cycle in $G_D$, contradicting
Claim~H or Claim~N or Claim~O. 
This completes the proof of Claim~P and therefore also of the theorem.
\end{proof}

%\begin{corollary}
%  If $D$ is a connected oriented graph of order $n$ and with maximum degree at most $4$, then $\fvs(D) \leq \frac{3n}{7}$.
%\end{corollary}

	\subsection{Digraphs}\label{4direct}
Let $\mathcal{C}_o$ denote the class of digraphs obtained from an undirected odd cycle by replacing each edge with a directed 2-cycle. For a feedback arc set $F$, let $c(F)$ denote the number of odd cycles contained in $D[F]$. We have the following:

\begin{theorem}\label{thm:fvs4di}
Let $D$ be a connected digraph of order $n$ with $\Delta(D) \le 4$. Then $\fvs(D) \le \frac{n}{2}$ unless $D \in \mathcal{C}_o$.
	\end{theorem}

\begin{proof}
Let $D\notin \mathcal{C}_o$ and $F\subseteq A(D)$ be a minimum feedback arc set of $D$. 
\begin{claim}\label{claim41}
			$\Delta(D[F]) \le 2$ and if $d_{D[F]}(v)=2$, then $d^{+}_{D}(v)=d^{-}_{D}(v)=2$. 
		\end{claim}
		\begin{proof}
We first show that $\Delta(D[F]) \le 2$. Assume for contradiction that there exists a
vertex $v$ in $D[F]$ with total degree $d_{D[F]}(v)=d_{D[F]}^-(v)+d_{D[F]}^+(v)\ge 3$. Since $\Delta(D)\le4$, 
$\min\{d^+(v),d^-(v)\}\le 2 < 3$. Without loss of generality, assume that $d^-(v)\le d^+(v)$. Construct an arc set $F'$ from $F$ by replacing the out-arcs of $v$ with in-arcs of $v$. Since every cycle containing $v$ must include at least one out-arc of $v$, $F'$ is a feedback arc set of $D$. Since $d^-(v)\le 2 < 3$ and $d_{D[F]}(v)\ge 3$, we have $|F'|<|F|$, contradicting the minimality of $F$.

Now we prove the second part of the claim. As $\Delta(D) \le 4$, it suffices to show $\min\{d^+(v), d^-(v)\} > 1$. Suppose for contradiction that there exists a
vertex $v$ in $D[F]$ with $\min\{d^+(v),d^-(v)\}\le 1$.  Without loss of generality, assume that $d^-(v)\le d^+(v)$. 
Construct an arc set $F''$ from $F$ by replacing the out-arcs of $v$ with in-arcs of $v$. Since every cycle containing $v$ must include at least one out-arc of $v$, $F''$ is a feedback arc set of $D$. Since $d^-(v)\le 1 < 2$ and $d_{D[F]}(v)=2$, we have $|F''|<|F|$, a contradiction to the minimality of $F$.
\end{proof}
Let $F^*\subseteq A(D)$ be a minimum feedback arc set of $D$ with minimum $c(F^*)$. 
\begin{claim}\label{claim2}
			$D[F^*]$ is bipartite. 
		\end{claim}
\begin{proof}
Suppose that there exists an odd cycle $C=x_1,x_2,\cdots,x_{2k+1}$. Clearly, there exists a vertex $x_i$ such that $x_{i-1}x_i,x_ix_{i+1}\in A(D[F^*])$. By Claim~\ref{claim41}, there is a vertex $x_i^- \neq x_{i-1}$ such that $x_i^-x_i\in A(D)$. 
Now replace ${x_ix_{i+1}}$ with ${x_i^-x_i}$ to obtain another minimum feedback arc set, say $F'$. Note that if $x_i^-\in V(D\setminus C)$, then $c(F')<c(F)$, it is a contradiction. So, we may assume $x_{i^-}\in V(C)$. 
If $x_i^-\in V(C)\setminus\{x_i,x_{i-1},x_{i+1}\}$, then $d_{D[F']}(x_i^-)=3$, contradicting $\Delta(D[F']) \le 2$. 
Thus, $x_i^-=x_{i+1}$ and $x_{i+1}x_i\in A(D)$. By symmetry, $x_ix_{i-1}\in A(D)$. Hence, $x_i$ is incident to two 2-cycles in $D$. If $x_{i+2}x_{i+1}\in A(C)$, then $F_1=(F^*\setminus \{x_ix_{i+1}\})\cup \{x_{i+1}x_i\}$ is a minimum feedback arc set of $D$ with $c(F_1)=c(F^*)$. Similarly, since $x_{i+2}x_{i+1},x_{i+1}x_{i}\in A(D[F_1])$, $x_{i+1}$ is incident to two 2-cycles in $D$. If $x_{i+1}x_{i+2}\in A(D)$, then $x_ix_{i+1}, x_{i+1}x_{i+2}\in A(D)$, by the same argument we have $x_{i+1}$ is incident to two 2-cycles in $D$. And we keep doing like above, so for all $x_i\in V(C)$, $i\in [2k+1]$, $x_i$ is incident to two 2-cycles in $D$. Hence, $D\in \mathcal{C}_o$, a contradiction.
\end{proof}

Since $D[F^*]$ is bipartite, one of the parts, say $X$, contains at most half of the vertices. So, $\fvs(D) \le  \frac{|V(D)|}{2} $.
\end{proof}
        
This theorem is tight due to a digraph of order $5$ with feedback vertex set equal to 2, (see Figure~\ref{digraph4}).
    
\begin{figure}[htb]
\begin{center}
\tikzstyle{vertexX}=[circle, draw, top color=black!20, bottom color=black!10, minimum size=10pt, scale=0.75, inner sep=0.8pt]    
\begin{tikzpicture}[>=Stealth, thick, 
    every node/.style={circle,draw,minimum size=8mm,inner sep=1pt}]

\node (v1) at (90:2)  {$v_1$};
\node (v2) at (18:2)  {$v_2$};
\node (v3) at (-54:2) {$v_3$};
\node (v4) at (-126:2){$v_4$};
\node (v5) at (-198:2){$v_5$};

\draw[->,bend left=15]  (v1) to (v2); 

\draw[->,bend left=15]  (v2) to (v3);
\draw[->,bend left=15]  (v3) to (v2);

\draw[->,bend left=15]  (v3) to (v4);
\draw[->,bend left=15]  (v4) to (v3);

\draw[->,bend left=15]  (v4) to (v5);
\draw[->,bend left=15]  (v5) to (v4);

\draw[->,bend left=15]  (v5) to (v1);
\draw[->,bend left=15]  (v1) to (v5);

\end{tikzpicture}
\end{center}
\caption{a digraph of order $5$ with feedback vertex set equal to 2}
\label{digraph4}
\end{figure}
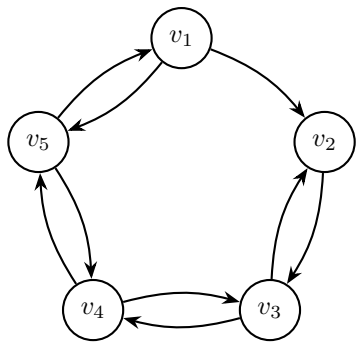
\noindent

\section{Oriented and Directed Graphs with $\Delta\leq 5$}

In Subsections \ref{sec:5or} and \ref{sec:5di}, we consider orgraphs and digraphs, respectively.

\subsection{Orgraphs}\label{sec:5or}
\begin{lemma}\label{lemma1}
     Let $D$ be a connected orgraph with $\Delta(D) \le 5$. Then $D$ has a feedback arc set $F$ such that $D[F]$ is bipartite.
    \end{lemma}

\begin{proof}
Let $F\subseteq A(D)$ be a minimum feedback arc set of $D$. 
\begin{claim}\label{claim1}
			$\Delta(D[F]) \le 2$. 
		\end{claim}
		\begin{proof}
Assume for contradiction that there exists a
vertex $v$ in $D[F]$ with total degree $d_{D[F]}(v)=d_{D[F]}^-(v)+d_{D[F]}^+(v)\ge 3$. Since $\Delta(D)\le5$, 
$\min\{d^+(v),d^-(v)\}\le 2 < 3$. Without loss of generality, we assume that $d^-(v)\le d^+(v)$. Construct an arc set $F'$ from $F$ by replacing the out-arcs of $v$ with in-arcs of $v$. Since every cycle containing $v$ must include at least one out-arc of $v$, $F'$ is a feedback arc set of $D$. Since $d^-(v)\le 2 < 3$ and $d_{D[F]}(v)\ge 3$, we have $|F'|<|F|$, contradicting the minimality of $F$.
\end{proof}
If $D[F]$ contains no odd cycle, we are done. Thus, assume there exists an odd cycle $C$. Clearly, there exists a vertex $v\in V(D[F])$ with $d_{D[F]}^-(v)=d_{D[F]}^+(v)=1$. Replace the two incident arcs with in-arcs of $v$ if $d^-(v)\le d^+(v)$, otherwise with out-arcs of $v$. 

We claim that this replacement creates no new odd cycles, and the resulting set is still a minimum feedback arc set. Without loss of generality, suppose we replace ${vu}$ with ${u'v}$. Clearly, the resulting arc set, denoted by $F'$, remains a minimum feedback arc set. Thus, we only need to prove that this replacement does not create new odd cycles. Suppose, for contradiction, that the replacement creates a new odd cycle $C' \subseteq D[F']$, then $u'v\in A(C')$. Since $F'$ is still a minimum feedback arc set, Claim~\ref{claim1} implies that $\Delta(D[F']) \le 2$. Thus, $C'$ must contain the vertex $u$. However, $d_{D[F']}(u)=1$, a contradiction.

Repeat this replacement until the resulting set induces no odd cycles. Let the final set be denoted by $F''$, then $F''$ is a feedback arc set such that $D[F'']$ is bipartite.
\end{proof}

\begin{theorem}\label{thm:fvs5o}
Let $D$ be an orgraph of order $n$ with $\Delta(D) \le 5$. Then $\fvs(D) \le \frac{n}{2}$.
	\end{theorem}

\begin{proof}
We may assume without loss of generality that $D$ is connected. By Lemma~\ref {lemma1}, there exists a feedback arc set $F$ such that $D[F]$ is bipartite. Since $D[F]$ is bipartite, one of the parts, say $X$, contains at most half of the vertices. It is enough to show that $X$ is a feedback vertex set of $D$. Observe that $D\setminus X \subseteq D\setminus F$, since all arcs in $F$ have at least one endpoint in $X$. $D\setminus F$ is acyclic by definition of $F$. Then $D\setminus X$ is acyclic. Therefore, $X$ is a feedback vertex set.
\end{proof}

This theorem is tight due to a orgraph of order $6$ with feedback vertex set equal to 3, see Figure~\ref{orgraph5}.
    
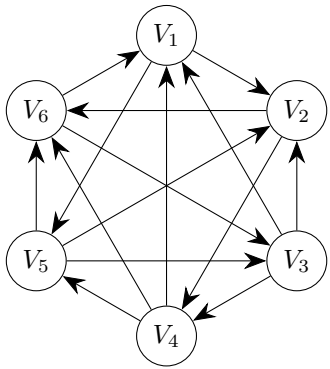
\begin{figure}[htb]
\tikzstyle{vertexX}=[circle, draw, top color=black!20, bottom color=black!10, minimum size=10pt, scale=0.75, inner sep=0.8pt] 
\begin{center}
\begin{tikzpicture}[scale=1, every node/.style={circle, draw, minimum size=0.3cm}]

\node (V1) at (0,2) {$V_1$};
\node (V2) at (1.732,1) {$V_2$};
\node (V3) at (1.732,-1) {$V_3$};
\node (V4) at (0,-2) {$V_4$};
\node (V5) at (-1.732,-1) {$V_5$};
\node (V6) at (-1.732,1) {$V_6$};

\foreach \i/\j in {V1/V2,V3/V1,V4/V1,V1/V5,V6/V1,
                    V3/V2,V2/V4,V5/V2,V2/V6,
                    V3/V4,V5/V3,V6/V3,
                    V4/V5,V4/V6,
                    V5/V6} {
    \draw[-{Stealth[length=3mm]}] (\i) -- (\j);
}
\end{tikzpicture}
\end{center}
\caption{an orgraph of order $6$ with feedback vertex set equal to 3}
\label{orgraph5}
\end{figure}
\noindent

\subsection{Digraphs}\label{sec:5di}

\begin{theorem}\label{thm:fvs5di}
Let $D$ be a digraph of order $n$ with $\Delta(D) \le 5$. Then $\fvs(D) \le \frac{2n}{3}$.
	\end{theorem}
\begin{proof}
Without loss of generality, we may assume that $D$ is connected. 
Let $F \subseteq A(D)$ be a minimum feedback arc set of $D$. 
By Claim~\ref{claim1}, we have $\Delta(D[F]) \le 2$. 
Consider a vertex cover of $D[F]$. 
In the worst case, $D[F]$ is a disjoint union of oriented triangles. Therefore, a minimum vertex cover of $D[F]$ has size at most 
$ \frac{2n}{3}$. Since every vertex cover of $D[F]$ is also a feedback vertex set of $D$, we are done.
\end{proof}

This theorem is tight due to a digraph of order $6$ with feedback vertex set equal to 4, see Figure~\ref{digraph5}.
\begin{figure}[htb]
\tikzstyle{vertexX}=[circle, draw, top color=black!20, bottom color=black!10, minimum size=10pt, scale=0.75, inner sep=0.8pt] 
\begin{center}
\begin{tikzpicture}[>=Stealth, thick,
    every node/.style={circle,draw,minimum size=8mm,inner sep=1pt,font=\small}]

\node (v1) at (-3,1.2) {$v_1$};
\node (v2) at (-3,-1.2) {$v_2$};
\node (v3) at (-1.2,0) {$v_3$};

\node (u1) at (3,1.2) {$u_1$};
\node (u2) at (3,-1.2) {$u_2$};
\node (u3) at (1.2,0) {$u_3$};

\foreach \a/\b in {v1/v2, v2/v3, v3/v1}{
  \draw[->,bend left=15] (\a) to (\b);
  \draw[->,bend left=15] (\b) to (\a);
}

\foreach \a/\b in {u1/u2, u2/u3, u3/u1}{
  \draw[->,bend left=15] (\a) to (\b);
  \draw[->,bend left=15] (\b) to (\a);
}

\foreach \a/\b in {1/1,2/2,3/3}{
  \draw[->,thick] (v\a) -- (u\b);
}

\end{tikzpicture}
\end{center}
\caption{a digraph of order $6$ with feedback vertex set equal to 4}
\label{digraph5}
\end{figure}
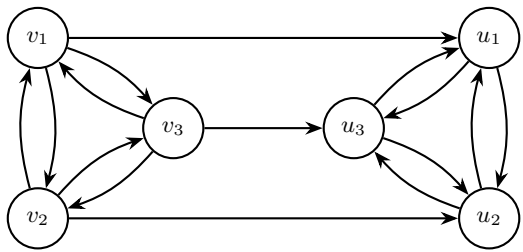

\noindent

\section{Conclusions}

For any integer $k \ge 2$, let $f(k)$ denote the supremum of $\fvs(D)/|V(D)|$ taken over all orgraphs $D$ with maximum degree at most $k$. In this paper we proved that $f(4)=\tfrac{3}{7}$ and $f(5)=\tfrac{1}{2}$. In general, it would be interesting to determine the value of $f(k)$ for larger $k$, or even to understand the asymptotic behaviour of the function $f$. We propose the following conjecture.

\begin{conjecture}
	There exists a constant $\theta>0$ such that for all $k\ge 3$,
	\[
	f(k)\le 1 - \theta\cdot \frac{\log_2 k}{k}.
	\]
\end{conjecture}

Recall that for every integer $n$ there exists a tournament $T$ on $n$ vertices whose largest acyclic subtournament has size at most $2\lfloor \log_2 n\rfloor + 1$. Consequently,
\[
\frac{\fvs(T)}{|V(T)|}
\ge
1 - \frac{2\lfloor \log_2 n\rfloor +1}{n}.
\]
Note that $k=n-1$ and $1-\frac{2\lfloor \log_2 (k+1)\rfloor +1}{k+1}\ge 1-\frac{2\log_2 k}{k}$ for $k$ large enough.
Thus, if true, the upper bound in the conjecture would be the best possible. 

We also have the following conjecture for $k=6$. 
\begin{conjecture}
$f(6)=\frac{4}{7}$. 
\end{conjecture}

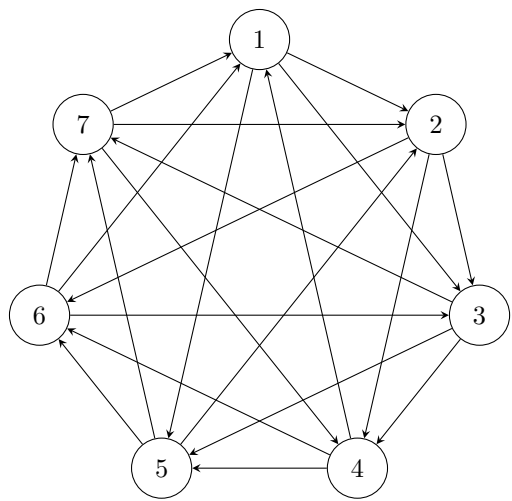
\begin{figure}
\begin{center}
\begin{tikzpicture}[>=stealth,
    every node/.style={circle, draw, minimum size=8mm, inner sep=0pt}]
    % Place the 7 vertices on a circle of radius 3
    \foreach \i in {1,...,7} {
        \node (v\i) at ({90-360/7*(\i-1)}:3cm) {$\i$};
    }

    % Draw directed edges: N^+(i) = {i+1, i+2, i+4} mod 7
    \foreach \i in {1,...,7} {
        \foreach \d in {1,2,4} {
            \pgfmathtruncatemacro{\j}{mod(\i-1+\d,7)+1}
            \draw[->] (v\i) -- (v\j);
        }
    }
\end{tikzpicture}\label{PalT7}\caption{Paley tournament on 7 vertices}
\end{center}
\end{figure}

Note that $f(6)\geq \frac{4}{7}$ as the minimum feedback vertex set of the following Paley tournament $T_7$ on 7 vertices (see Figure 17) %\ref{PalT7}
has 4 vertices (since $T_7[N^+(v)]$ is a 3-cycle for all $v\in V(T_7)$ and therefore $T_7$ contains no transitive subtournament on four vertices).

% [1] B. Reed, Paths, stars, and the number three, Combin. Probab. Comput. 5 (1996) 277–295. 

\bibliography{bib}

\end{document}